\documentclass[oneside,english]{amsart}
\usepackage[T1]{fontenc}
\usepackage[latin9]{inputenc}
\usepackage{geometry}
\usepackage{verbatim}
\usepackage{mathrsfs}
\usepackage{amsthm}
\usepackage{amssymb}

\makeatletter
\numberwithin{equation}{section}
\numberwithin{figure}{section}
\theoremstyle{plain}
\newtheorem{thm}{\protect\theoremname}
  \theoremstyle{remark}
  \newtheorem{rem}[thm]{\protect\remarkname}
  \theoremstyle{plain}
  \newtheorem{lem}[thm]{\protect\lemmaname}
  \theoremstyle{remark}
  \newtheorem*{claim*}{\protect\claimname}

\usepackage{tikz}
\usepackage{xcolor}

\makeatother

\usepackage{babel}
  \providecommand{\claimname}{Claim}
  \providecommand{\lemmaname}{Lemma}
  \providecommand{\remarkname}{Remark}
\providecommand{\theoremname}{Theorem}

\begin{document}

\title{Metastability for the Ising Model on the hypercube }

\author{Oliver Jovanovski}
\begin{abstract}
We consider Glauber dynamics for the low-temperature, ferromagnetic
Ising Model set on the $n$-dimensional hypercube. We derive precise
asymptotic results for the crossover time (the time it takes for the
dynamics to go from the configuration with a $"-1"$ at every vertex,
to the configuration with a $"+1"$ at each vertex) in the limit as
the inverse temperature $\beta\rightarrow\infty$. 
\end{abstract}
\maketitle

\section{\label{sec:intro}Ising Model on the hypercube}

Put in simple terms, metastability is the phenomenon describing a
stochastic process that is temporarily trapped in the neighbourhood
of a state other than the 'most stable' state. Usually this 'trap'
comes in the form of a local minimum of an associated energy function,
and over a short time scale the observed process appears to be in
a quasi-equilibrium. Viewed over a longer time scale, the process
manages (after many unsuccessful attempts) to overcome the energy
barrier that separates it from a global minimum, which is often unique
and its only true equilibrium. 

Observations of this phenomenon in the physical world are abundent
(see for example \cite{key-3}, \cite{key-4}, \cite{key-5}), and
thus to no surprise many mathematical models have been used to study
it (e.g. \cite{key-6}, \cite{key-7}). A notable example of this
is the Ising model set on a finite subset of $\mathbb{Z}^{d}$ in
the low-temperature regime. This has been well studied, and precise
results for the \emph{crossover time} (i.e. the time it takes for
the dynamics to go from the configuration that assigns a $"-1"$ to
every vertex, to the configuration that assigns a $"+1"$ to every
vertex) have been derived in the $\mathbb{Z}^{2}$ and $\mathbb{Z}^{3}$
setting (see Chapters 16-20 in \cite{key-2} for overview). In this
paper we will derive similar results for when the setting is an $n$-dimensional
hypercube. We do this by employing tools developed in \cite{key-2}
and determining the geometric properties of the hypercube required
to use these tools. A priori, one might expect to see a significantly
slower crossover on the hypercube, compared to a rectangle in $\mathbb{Z}^{d}$
with the same number of vertices. Indeed, we will show that on the
hypercube the crossover time depends strongly on the size of the graph
due to the expander properties of graph, while in the latter there
is only a weak dependence on this.

\medskip{}

We will denote the graph of the $n$-dimensional hypercube by $\mathcal{Q}_{n}=\left(V_{n},\, E_{n}\right)$,
where $V_{n}=\left\{ 0,1\right\} ^{n}$ are its vertices and $E_{n}\,:=\left\{ \left(v,w\right)\in V_{n}\times V_{n}\,:\,\left\Vert v-w\right\Vert _{1}=1\right\} $
its edges. If $\mathcal{Q}_{r}$ is an $r$-dimensional \emph{sub-cube}
of $\mathcal{Q}_{n}$ (a subgraph of size $2^{r}$ that is isomorphic
to an $r$-dimensional hypercube, and hence with all its vertices
agreeing on $n-r$ co-ordinates), we shall (by a minor abuse of notation)
write $"A\subseteq\mathcal{Q}_{r}"$ to mean that $A$ is a subset
of the vertices in $\mathcal{Q}_{r}$. By \emph{``Ising Model on
the hypercube''} we are referring to the \emph{configuration space}
$\Omega:=\left\{ +1,\,-1\right\} ^{V_{n}}$ together with an associated
Gibbs measure on this space, defined in (\ref{eq:gibbsmeasure}).
This configuration space corresponds to the assignment to each vertex
of exactly one of two spins (either $+1$ or $-1$). Hence an equivalent
representation of $\Omega$ is the power set $\mathcal{P}\left(V_{n}\right)$
of $V_{n}$, where $A\in\mathcal{P}\left(V_{n}\right)$ is identified
with the configuration that assigns $\left(+1\right)$ to every vertex
in $A$, and $\left(-1\right)$ to every vertex in $\overline{A}$
(the complement of $A$). Therefore we will (by further abuse of notation)
identify $\Omega$ with $\mathcal{P}\left(V_{n}\right)$ and refer
to the terms in $\mathcal{P}\left(V_{n}\right)$ (and hence $\Omega$)
as configurations, whenever there is no threat of ambiguity. 

Two special configurations (subsets) deserve their own symbols - we
will denote by $\boxplus$ and $\boxminus$ the configurations $V_{n}$
and $\emptyset$ in $\Omega$ (equivalently, these are the two configurations
with a $\left(+1\right)/\left(-1\right)$ assigned to every vertex).
The Hamiltonian function $\mathcal{H}:\Omega\rightarrow\mathbb{R}$
associates an energy with each configuration $A\in\Omega$ according
to 
\begin{equation}
\mathcal{H}\left(A\right):=-\frac{\mathfrak{J}}{2}\left(\left|E_{n}\right|-2\left|E\left(A,\overline{A}\right)\right|\right)-\frac{\mathfrak{h}}{2}\left(\left|A\right|-\left|\overline{A}\right|\right)\label{eq:hamiltonian}
\end{equation}
where for two subsets $U,W\subseteq V_{n}$, $E\left(U,W\right)\subseteq E_{n}$
is the set of all unoriented edges with one endpoint in $U$ and another
in $W$, and $\mathfrak{J}>0$, $\mathfrak{h}\in\mathbb{R}$ are fixed
constants, known as the \emph{interaction} and \emph{external field}
parameters, respectively. The Gibbs probability measure on $\Omega$
is given by 
\begin{equation}
\mu_{\beta}\left(A\right)=\frac{1}{Z_{n}}\exp\left(-\beta\mathcal{H}\left(A\right)\right)\label{eq:gibbsmeasure}
\end{equation}
with $\beta\geq0$ being the \emph{inverse temperature} and $Z_{n}$
the normalizing constant. Our interest is restricted to the limit
$\beta\rightarrow\infty$, thus we may take $\mathfrak{J}=1$, which
simply corresponds to a rescaling of $\beta$ and $\mathfrak{h}$.
Then with $\mathfrak{J}=1$ in (\ref{eq:hamiltonian}), we will in
addition also assume that $0<\mathfrak{h}<n$ is not of the form $\frac{a}{b}$
for some $a\in\mathbb{N}$ and $b\in\left\{ 1,2,\ldots,2^{n}\right\} $,
which will simplify much of our analysis and avoid certain degeneracies
(note that we are only excluding a finite number of real values from
the interval $\left[0,n\right]$). It is evident that if $\mathfrak{h}\geq n$,
then $\boxminus$ is a global maximum of $\mathcal{H}$, and any path
$\gamma$ of minimal length (equal to $2^{n}+1$) going from $\boxminus$
to $\boxplus$, is monotone decreasing in $\mathcal{H}$. Hence there
is a drift towards $\boxplus$, and no metastability would arise in
such a model.

The final ingredient will be to define the dynamics on $\Omega$.
For this, we consider continuous-time Glauber dynamics, which is a
reversible, continuous-time Markov process $\left(\xi_{t}\right)_{t\geq0}$
with (\ref{eq:gibbsmeasure}) as its equilibrium measure, and is defined
by the transition rates

\begin{equation}
c_{\beta}\left(\xi,\xi\prime\right)=\begin{cases}
\exp\left(-\beta\left[\mathcal{H}\left(\xi\prime\right)-\mathcal{H}\left(\xi\right)\right]_{+}\right), & \left(\xi,\xi\prime\right)\in\mathcal{E}_{n}\\
0 & \mbox{otherwise}
\end{cases}\label{eq:transitionrates}
\end{equation}
where $\left[\mathcal{H}\left(\xi\prime\right)-\mathcal{H}\left(\xi\right)\right]_{+}:=\max\left\{ 0,\mathcal{H}\left(\xi\prime\right)-\mathcal{H}\left(\xi\right)\right\} $
and $\mathcal{E}_{n}:=\left\{ \left(A,A\prime\right)\in\mathcal{P}\left(V_{n}\right)\times\mathcal{P}\left(V_{n}\right)\,:\,\left|A\bigtriangleup A\prime\right|=1\right\} $
can be thought of as edges on the configuration space. With these
definitions in mind, we can now state our main results. 
\begin{thm}
\label{thm:expcrossovertime}For the Markov process $\left(\xi_{t}\right)_{t\geq0}$
with transition rates give by (\ref{eq:transitionrates}), let $\tau_{\boxplus}$
be the hitting time of the state $\boxplus$. Then 
\[
\lim_{\beta\rightarrow\infty}\exp\left(-\beta\Gamma^{\dagger}\right)\mathbb{E}_{\boxminus}\left[\tau_{\boxplus}\right]=K
\]
where 
\begin{eqnarray*}
\Gamma^{\dagger} & = & \frac{1}{3}\left(2-\mathfrak{h}+\left\lfloor \mathfrak{h}\right\rfloor \right)\left(2^{\left\lceil n-\mathfrak{h}\right\rceil }-4+2\epsilon\right)-\epsilon\\
K & = & \frac{\left\lceil \mathfrak{h}\right\rceil !}{n!2^{n-4}\left(3-\epsilon\right)}
\end{eqnarray*}
and $\epsilon=1-\left\lfloor n-\mathfrak{h}\right\rfloor \mbox{mod}2$. 

\medskip{}
\end{thm}
\begin{rem}
The exponent $\Gamma^{\dagger}$ scales proportionally to the size
of the underlying graph. Indeed, as $n\rightarrow\infty$, we get
that $\Gamma^{\dagger}/\left|\mathcal{Q}_{n}\right|\rightarrow2^{-\left\lfloor \mathfrak{h}\right\rfloor }\left(2-\mathfrak{h}+\left\lfloor \mathfrak{h}\right\rfloor \right)/3$.
This agrees with the expander property of the hypercube, which tells
us that the term $\left|E\left(A,\overline{A}\right)\right|$ in (\ref{eq:hamiltonian})
will grow proportional to $\left|A\right|$, for all $A$ up to size
$\left|A\right|\leq2^{n-1}$. 
\end{rem}
\medskip{}

Theorem \ref{thm:expcrossovertime} is an application of Theorem 16.5
in \cite{key-2}. To do this effectively, we need to compute the \emph{potential-barrier
height} between $\boxminus$ and $\boxplus$ (defined in (\ref{eq:gammastar})),
represented by $\Gamma^{\dagger}$ in the above theorem. The prefactor
$K$ in Theorem \ref{thm:expcrossovertime} is based on a variational
problem given in Lemma 16.17 in \cite{key-2} (and also stated below
in equation (\ref{eq:variationalform})), and will be solved for our
problem in Section \ref{sec:Critical-and-protocritical}. Furthermore,
Theorem 16.5 is subject to hypothesis $(H1)$ in (\ref{eq:H2}), and
the validity of this will be verified in Theorem \ref{lem:metastabset}. 

\medskip{}
An important property in this model will be the \emph{communication
height} between two configurations $\xi$, $\xi\prime$, defined by
\begin{equation}
\Phi\left(\xi,\xi\prime\right)=\min_{\gamma:\xi\rightarrow\xi\prime}\max_{\sigma\in\gamma}\mathcal{H}\left(\sigma\right)\label{eq:commheight}
\end{equation}
where the minimum is taken over all paths $\gamma:\xi\rightarrow\xi\prime$
moving along the edge set $\mathcal{E}_{n}$. We also define the \emph{stability
level} of $\xi\in\Omega$ by 
\begin{equation}
\mathscr{V}_{\xi}=\min_{\zeta:\,\mathcal{H}\left(\zeta\right)<\mathcal{H}\left(\xi\right)}\Phi\left(\xi,\zeta\right)-\mathcal{H}\left(\xi\right)\label{eq:stablevel}
\end{equation}
It is easy to see from the definition of $\mathcal{H}$ in (\ref{eq:hamiltonian})
that set of \emph{stable configurations, $\Omega_{s}:=\left\{ \xi\in\Omega:\,\mathcal{H}\left(\xi\right)=\min_{\xi\in\Omega}\mathcal{H}\left(\xi\right)\right\} $},
always reduces to $\Omega_{s}=\left\{ \boxplus\right\} $. The set
of \emph{metastable configurations }is defined by 
\[
\Omega_{m}=\left\{ \xi\in\Omega\backslash\left\{ \boxplus\right\} :\,\mathscr{V}_{\xi}=\max_{\xi\in\Omega\backslash\left\{ \boxplus\right\} }\mathscr{V}_{\xi}\right\} 
\]
and generally it is not a trivial task to determine which configurations
belong in $\Omega_{m}$. Thus, the following theorem is an important
prerequisite to all further analysis.
\begin{thm}
\label{lem:metastabset}For the Ising Model on $\mathcal{Q}_{n}$,
$\Omega_{m}=\left\{ \boxminus\right\} $. 
\end{thm}
The proof of this theorem is given in Section \ref{sec:stablevelrefpath}.

\medskip{}

Now given Theorem \ref{lem:metastabset}, we define the \emph{potential-barrier
height }between the metastable and stable configurations by
\begin{equation}
\Gamma^{\star}=\Phi\left(\boxminus,\boxplus\right)-\mathcal{H}\left(\boxminus\right)\label{eq:gammastar}
\end{equation}
Note from (\ref{eq:hamiltonian}) that for any $\sigma\in\Omega$
(recall that we are taking $\mathfrak{J}=1$),
\begin{eqnarray*}
\mathcal{H}\left(\sigma\right)-\mathcal{H}\left(\boxminus\right) & = & -\frac{\mathfrak{J}}{2}\left(\left|E_{n}\right|-2\left|E\left(\sigma,\overline{\sigma}\right)\right|\right)-\frac{\mathfrak{h}}{2}\left(\left|\sigma\right|-\left|\overline{\sigma}\right|\right)+\frac{\mathfrak{J}}{2}\left(\left|E_{n}\right|-2\left|E\left(\emptyset,V\right)\right|\right)+\frac{\mathfrak{h}}{2}\left(\left|\emptyset\right|-n\right)\\
 & = & \left|E\left(\sigma,\overline{\sigma}\right)\right|-\mathfrak{h}\left|\sigma\right|
\end{eqnarray*}
and hence 
\begin{equation}
\Gamma^{\star}=\min_{\gamma:\boxminus\rightarrow\boxplus}\max_{\sigma\in\gamma}\left(\left|E\left(\sigma,\overline{\sigma}\right)\right|-\mathfrak{h}\left|\sigma\right|\right)\label{eq:formulation2gammastar}
\end{equation}
We will call paths $\gamma:\,\boxminus\rightarrow\boxplus$ that satisfy
the minmax in (\ref{eq:formulation2gammastar}) \emph{optimal paths. }

One further point of interest will be the \emph{critical set $\mathscr{C}^{\star}\subseteq\Omega$
}and the \emph{proto-critical set} and $\mathscr{P}^{\star}\subseteq\Omega$,
defined as the unique, maximal subset $\mathscr{C}^{\star}\times\mathscr{P}^{\star}\subseteq\Omega^{2}$\emph{
}that satisfies the conditions 
\begin{eqnarray}
1. &  & \forall\xi\in\mathscr{P}^{\star},\,\exists\xi\prime\in\mathscr{C}^{\star}\mbox{ s.t. }\left(\xi,\xi\prime\right)\in\mathcal{E}_{n}\nonumber \\
2. &  & \forall\xi\in\mathscr{P}^{\star},\,\Phi\left(\xi,\mathbf{\boxminus}\right)<\Phi\left(\xi,\boxplus\right)\nonumber \\
3. &  & \forall\xi\in\mathscr{C}^{\star},\,\exists\gamma:\xi\rightarrow\boxplus\,\mbox{ s.t.}\,\max_{\zeta\in\gamma}\mathcal{H}\left(\zeta\right)-\mathcal{H}\left(\mathbf{\boxminus}\right)\leq\Gamma^{\star}\label{eq:critandprot}\\
 &  & \mbox{and}\,\gamma\cap\left\{ \zeta\in\Omega:\,\Phi\left(\zeta,\mathbf{\boxminus}\right)<\Phi\left(\zeta,\mathbf{\boxplus}\right)\right\} =\emptyset\nonumber 
\end{eqnarray}
Uniqueness follows from the observation that if $\left(\mathscr{C}_{1}^{\star},\mathscr{P}_{1}^{\star}\right)$
and $\left(\mathscr{C}_{2}^{\star},\mathscr{P}_{2}^{\star}\right)$
both satisfy the above conditions, then so does $\left(\mathscr{C}_{1}^{\star}\cup\mathscr{C}_{2}^{\star},\mathscr{P}_{1}^{\star}\cup\mathscr{P}_{2}^{\star}\right)$. 

To apply the tools developed in \cite{key-2} , we need to verify
the two hypotheses 

\begin{eqnarray}
(H1)\quad &  & \Omega_{m}=\left\{ \boxminus\right\} \nonumber \\
(H2)\quad &  & \xi\rightarrow\left|\xi\prime\in\mathscr{P}^{\star}:\,\xi\sim\xi\prime\right|\mbox{ is constant on }\mathscr{C}^{\star}\label{eq:H2}
\end{eqnarray}
Hypothesis $(H1)$ follows from Theorem \ref{lem:metastabset}, and
is the only one in (\ref{eq:H2}) that is necessary for the proof
of Theorem \ref{thm:expcrossovertime}. We will verify the validity
of $(H2)$ in Section (\ref{sec:Critical-and-protocritical}), where
we also derive a description of the sets $\mathscr{P}^{\star}$ and
$\mathscr{C}^{\star}$ defined in (\ref{eq:critandprot}). This will
permit us to conclude the result of Theorem 16.4 in \cite{key-2},
given by

\medskip{}

\begin{thm}
\label{thm:secondthm}~

(a) $\lim_{\beta\rightarrow\infty}\mathbb{P}_{\mathbf{\boxminus}}\left(\tau_{\mathscr{C}^{\star}}<\tau_{\boxplus}\left|\tau_{\boxplus}<\tau_{\mathbf{\boxminus}}\right.\right)=1$ 

(b) $\lim_{\beta\rightarrow\infty}\mathbb{P}_{\mathbf{\boxminus}}\left(\tau_{\mathscr{C}^{\star}}=x\right)=1/\left|\mathscr{C}^{\star}\right|$
for all $x\in\mathscr{C}^{\star}$

\medskip{}

\end{thm}
where for any $\mathcal{A}\subseteq\Omega$, 
\[
\tau_{\mathcal{A}}=\inf\left\{ t>0:\,\xi_{t}\in\mathcal{A},\,\exists0<s<t:\,\xi_{s}\neq\xi_{0}\right\} 
\]
is the first hitting time of the set $\mathcal{A}$ once the starting
configuration has been vacated. 
\begin{rem}
Theorems \ref{thm:expcrossovertime} and \ref{thm:secondthm} are
given in \cite{key-2} with the underlying graph being a finite subset
of a lattice. It is not hard to verify that the proofs of these theorems
do not rely on this lattice structure, and remain true for any graph.

\medskip{}

\end{rem}

\subsection{Outline of the paper}

It is evident from the setup of this problem (as described above)
that our main focus should be on particular geometric properties of
the hypercube. Section \ref{sec:Isoperimetric} deals with establishing
some known results related to isoperimetric inequalities on the hypercube,
and their relevance to our problem. In Section \ref{sec:Potential-barrier-height}
we supplement these with additional results on this subject and look
at local maxima of the function $\mathcal{H}$ in (\ref{eq:hamiltonian}),
to obtain the value of the potential-barrier height $\Gamma^{\star}$,
as defined in (\ref{eq:formulation2gammastar}). Sectoion \ref{sec:Critical-and-protocritical}
is devoted to computing the value of $K$ in Theorem \ref{thm:expcrossovertime},
while Section \ref{sec:stablevelrefpath} contains only a proof of
Theorem \ref{lem:metastabset}. In Appendix \ref{sec:Appendix-A}
we prove the converse of a result provided in \cite{key-1}, which
is required in our analysis of the sets $\mathscr{P}^{\star}$ and
$\mathscr{C}^{\star}$.

\section{\label{sec:Isoperimetric}Isoperimetric inequalities for the hypercube}

The definitions (\ref{eq:hamiltonian}) and (\ref{eq:gammastar})
suggest that $\Gamma^{\star}$ will be closely related to edge-isoperimetric
properties of the graph $\mathcal{Q}_{n}$. Fortunately, such properties
have been well studied and known for some time (see \cite{key-1}).
In particular, the most relevant result for us involves identifying
the subsets of $\mathcal{Q}_{n}$ that have a minimal edge-boundary
over all subsets of some fixed size $k$. The following is a consequence
of Theorem 1.4 and 1.5 in \cite{key-1}: for $0<k<2^{n}$, a subset
$S\subseteq\mathcal{Q}_{n}$ with $\left|S\right|=k$ that has a minimal
edge-boundary (i.e. $\forall U\subseteq V_{n}\mbox{ of size }\left|U\right|=k,\,\left|E\left(U,\overline{U}\right)\right|\geq\left|E\left(S,\overline{S}\right)\right|$)
is given by 
\begin{equation}
\Upsilon_{k}=\left\{ v=\left(v_{1},\ldots,v_{n}\right)\in Q_{n}\,\vert\,\sum_{i=1}^{n}v_{i}2^{i-1}<k\right\} \label{eq:minimalset}
\end{equation}
and its edge-boundary is of size
\begin{equation}
\left|E\left(\Upsilon_{k},\overline{\Upsilon_{k}}\right)\right|=nk-2\sum_{i=1}^{k-1}q\left(i\right)\label{eq:minimalboundary}
\end{equation}
where $q\left(i\right)$ is the sum of all digits appearing in the
binary expansion of the number $i$. For a set $S$ of size $k$,
we will say that $\left|E\left(S,\overline{S}\right)\right|$ is \emph{minimal}
if $S$ satisfies the minimal edge-boundary condition in (\ref{eq:minimalboundary}). 
\begin{rem}
\label{remgoodset}In \cite{key-1}, \emph{good} subsets of $V_{n}$
are defined recursively as follows: $S\subseteq V_{n}$ with $\left|S\right|=k$
is called a \emph{good} set if $(a)$ $k=1$, or $(b)$ if $2^{r}<k\leq2^{r+1}$
for some $0\leq r\leq n-1$ and there is some $r+1$ dimensional sub-cube
$C_{r+1}$ containing $S$, such that $C_{r+1}$ decomposes into two
$r$-dimensional sub-cubes, $C_{r+1}=\left(C_{r}^{1},\, C_{r}^{2}\right)$,
which satisfy $\left|S\cap C_{r}^{1}\right|=2^{r}$ and $S\cap C_{r}^{2}$
is a good set. It is shown that if $S$ is a good set of size $k$,
then $\left|E\left(S,\overline{S}\right)\right|$ is minimal. Equivalently,
every good set $S$ makes$\left|E\left(S,S\right)\right|$ \emph{maximal
}(i.e. for any $U\subseteq V_{n}$ of size $k$, $\left|E\left(S,S\right)\right|\geq\left|E\left(U,U\right)\right|$).
It is easy to verify that (\ref{eq:minimalset}) defines a good set
for every $k$, and thus by symmetry, the set of all good sets is
the set of all images of (\ref{eq:minimalset}) under isomorphisms
of $\mathcal{Q}_{n}$. 
\end{rem}
It is obvious from the symmetries of the hypercube that any translation
of $\Upsilon_{k}$ by means of an isomorphism of $\mathcal{Q}_{n}$
will give a set with the same minimizing properties. In fact, by the
following lemma, these are all the sets with minimal edge-boundary.
\begin{lem}
\label{lem:allminimalsets}Let $S$ be a subset of the hypercube of
size $k$. Then $\left|E\left(S,\overline{S}\right)\right|$ is minimal
if and only if $S$ is some translation of the set $\Upsilon_{\left|S\right|}$
by an isomorphism of $\mathcal{Q}_{n}$. Equivalently, $\left|E\left(S,\overline{S}\right)\right|$
is minimal if and only if $S$ is a good set.
\end{lem}
While the knowledge that good sets have a minimal edge-boundary will
suffice in determining $\Gamma^{\star}$, Lemma (\ref{lem:allminimalsets})
will be important in Section \ref{sec:Critical-and-protocritical}
where we calculate the prefactor $K$ in Theorem \ref{thm:expcrossovertime}.
The proof of Lemma \ref{lem:allminimalsets} is given in Appendix
A. 

\medskip{}

Let $\Upsilon_{0}=\boxminus$, and note that the path $\gamma:\,\boxminus\rightarrow\boxplus$
given by 
\begin{equation}
\gamma=\left(\Upsilon_{0},\Upsilon_{1},\ldots,\Upsilon_{2^{n}-1},V_{n}\right)\label{eq:optpath}
\end{equation}
is a Glauber path (i.e. a path along the edge set $\mathcal{E}_{n}$),
since by definition the set $\Upsilon_{k+1}=\Upsilon_{k}\cup\left\{ w\right\} $
where $w=\left(w_{1},\ldots,w_{n}\right)\in\mathcal{Q}_{n}$ is the
unique vertex that satisfies $\sum_{i=1}^{n}w_{i}2^{i-1}=k+1$. Hence
we have the following immediate conclusion.
\begin{lem}
\label{lem:uniformoptimalpath}The path $\gamma$ in (\ref{eq:optpath})
is a uniformly optimal path. In other words, for all $0\leq i\leq2^{n}$
and for all $\sigma\in\Omega$ with $\left|\sigma\right|=i$, $\mathcal{H}\left(\sigma\right)\geq\mathcal{H}\left(\gamma_{i}\right)$. 
\end{lem}

\section{\label{sec:Potential-barrier-height}Potential-barrier height}

From Lemma \ref{lem:uniformoptimalpath} we know that the path $\gamma$
in (\ref{eq:optpath}) is an optimal path. In this section we will
determine the maximum value $\mathcal{H}$ attains along this path,
which by definition is equal to $\Gamma^{\star}$. 
\begin{lem}
\label{lem:communicationheight}The communication height $\Gamma^{\star}$
defined in (\ref{eq:gammastar}) is equal to 
\[
\Gamma_{n}^{\star}=\frac{1}{3}\left(2-\mathfrak{h}+\left\lfloor \mathfrak{h}\right\rfloor \right)\left(2^{\left\lceil n-\mathfrak{h}\right\rceil }-4+2\epsilon\right)-\epsilon
\]
where $\epsilon=1-\left\lfloor n-\mathfrak{h}\right\rfloor \mod2$.
\end{lem}
To prove Lemma \ref{lem:communicationheight}, we will first establish
a few elementary results.
\begin{lem}
\label{lem:2torsum}For any $0\leq r\leq n$, 
\begin{equation}
\sum_{i=1}^{2^{r}-1}q\left(i\right)=r2^{r-1}\label{eq:expansionofbinary}
\end{equation}
\end{lem}
\begin{proof}
Note that (\ref{eq:expansionofbinary}) is clearly true for $r\in\left\{ 0,1\right\} $.
Suppose that this also holds for all $r\in\left\{ 1,\ldots,k\right\} $.
Then 
\[
\sum_{i=1}^{2^{k+1}-1}q\left(i\right)=\sum_{i=1}^{2^{k}-1}q\left(i\right)+\sum_{i=2^{k}}^{2^{k+1}-1}q\left(i\right)=k2^{k-1}+2^{k}+\sum_{i=0}^{2^{k}-1}q\left(i\right)=\left(k+1\right)2^{k}
\]
The second equality follows from the observation that for any $0\leq i<2^{k}$,
the binary expansion of the number $2^{k}+i$ has exactly one more
$"1"$ than the binary expansion of the number $i$. 
\end{proof}
A different proof of Lemma \ref{lem:2torsum} is also given in \cite{key-1}. 
\begin{lem}
\label{lem:abdiff}Let $1\leq j<n-1$ and $1\leq a<2^{n}$, and let
the binary expansion of $a$ be given by $a=\sum_{i=1}^{n}a_{i}2^{i-1}$,
$a_{i}\in\left\{ 0,1\right\} $. Suppose also that $a_{j}=1$ and
$a_{j+1}=0$, and let $b=a+2^{j-1}$. Then 
\begin{equation}
\sum_{i=1}^{b-1}q\left(i\right)=\sum_{i=1}^{a-1}q\left(i\right)+\left(j+1+2\sum_{i=j+2}^{n}a_{i}\right)2^{j-2}\label{eq:abrelation}
\end{equation}
\end{lem}
\begin{proof}
Observe first that the binary expansion of $b$ is obtained from the
binary expansion of $a$ by switching $a_{j}$ with $a_{j+1}$. Now
suppose first that $a<2^{j}$, so that $a_{j}$ is the last $"1"$
appearing in the binary expansion of $a$. Then $a=2^{j-1}+c$ for
some $c<2^{j-1}$ and from Lemma \ref{lem:2torsum} it follows that
\[
\sum_{i=1}^{a-1}q\left(i\right)=\sum_{i=1}^{2^{j-1}-1}q\left(i\right)+\sum_{i=2^{j-1}}^{2^{j-1}+c-1}q\left(i\right)=\left(j-1\right)2^{j-2}+c+\sum_{i=0}^{c-1}q\left(i\right)
\]
while 
\[
\sum_{i=1}^{b-1}q\left(i\right)=j2^{j-1}+c+\sum_{i=1}^{c-1}q\left(i\right)=\sum_{i=1}^{a-1}q\left(i\right)+\left(j+1\right)2^{j-2}
\]
which agrees with (\ref{eq:abrelation}). We can now drop the assumption
$a<2^{j+1}$ by noting that each term in the sum $\sum_{i=a}^{b-1}q\left(i\right)$
(and there are $2^{j-1}$ such terms) has in its binary expansion
exactly $\sum_{i=j+2}^{n}a_{j}$ many $"1"$s beyond the $j+1^{st}$
term. 
\end{proof}
We can now proceed with a proof of Lemma \ref{lem:communicationheight}
\begin{proof}[Proof of Lemma \ref{lem:communicationheight} ]
 For $0\leq k\leq2^{n}$, define $g\left(k\right):=\left|E\left(\Upsilon_{k},\overline{\Upsilon_{k}}\right)\right|-\mathfrak{h}k$.
Then from (\ref{eq:formulation2gammastar}), (\ref{eq:minimalboundary})
and Lemma \ref{lem:uniformoptimalpath} it follows that 

\begin{eqnarray}
\Gamma^{\star} & = & \max_{0\leq k\leq2^{n}}g\left(k\right)\nonumber \\
 & = & \max_{0\leq k\leq2^{n}}\left\{ k\left(n-\mathfrak{h}\right)-2\sum_{i=1}^{k-1}q\left(i\right)\right\} \label{eq:gammastarmaxexpression}
\end{eqnarray}
The function $g$ is decreasing on $\left\{ k,k+1\right\} $ if and
only if $g\left(k+1\right)<g\left(k\right)$ which is equivalent to
\begin{equation}
2\left(\sum_{i=1}^{k}q\left(i\right)-\sum_{i=1}^{k-1}q\left(i\right)\right)=2q\left(k\right)>\left(n-\mathfrak{h}\right)\label{eq:qkgreaterthan}
\end{equation}
Notice that since $\mathfrak{h}$ is not an integer, (\ref{eq:qkgreaterthan})
must indeed be a strict inequality. Similarly, $g$ is increasing
on $\left\{ k-1,k\right\} $ if and only if $2q\left(k-1\right)<\left(n-\mathfrak{h}\right)$.
Therefore local maxima of $g$ occur at values $k$ that satisfy both
of the aforementioned conditions. By noting that $q\left(k\right)-q\left(k-1\right)\leq1$,
it follows that $k$ and $k-1$ must have exactly $\delta:=\left\lceil \left(n-\mathfrak{h}\right)/2\right\rceil $
and $\delta-1$ digits equal to $"1"$ in their binary expansion,
respectively. Hence, to determine the maximum value of $g$, it suffices
to consider values $k$ that satisfy these conditions. 

Observe also that if $k\geq2$ is even, then $q\left(k\right)\leq q\left(k-1\right)$,
hence we only need to consider odd $k$. Now suppose that $k^{\left(1\right)}$
is an integer that satisfies the above conditions, with its binary
expansion given by $k^{\left(1\right)}=\sum_{i=1}^{n}k_{i}^{\left(1\right)}2^{i-1}$.
Furthermore, suppose that $k_{j}^{\left(1\right)}=1$ and $k_{j+1}^{\left(1\right)}=0$
for some $j\geq1$. Let $k^{\left(2\right)}=k^{\left(1\right)}+2^{j-1}$,
so that the binary expansion of $k^{\left(2\right)}$ is obtained
from that of $k^{\left(1\right)}$ by switching $k_{j}^{\left(1\right)}$
with $k_{j+1}^{\left(1\right)}$. By Lemma \ref{lem:abdiff} we have
that 
\begin{eqnarray}
g\left(k^{\left(2\right)}\right)-g\left(k^{\left(1\right)}\right) & = & \left(k^{\left(2\right)}-k^{\left(1\right)}\right)\left(n-\mathfrak{h}\right)-2\left(\sum_{i=1}^{k^{\left(2\right)}-1}q\left(i\right)-\sum_{i=1}^{k^{\left(1\right)}-1}q\left(i\right)\right)\nonumber \\
 & = & 2^{j-1}\left(n-\mathfrak{h}-j-1-2\sum_{i=j+2}^{n}k_{i}^{\left(1\right)}\right)\label{eq:k2k1comp}
\end{eqnarray}
We can now use (\ref{eq:k2k1comp}) to compare the local maxima of
$g$ in order to find its global maximum. Starting with any $k=\sum_{i=1}^{n}k_{i}2^{i-1}$
that satisfies the aforementioned conditions ($k$ is odd, $k$ has
$\delta$ digits equal to $"1"$ in its binary expansion, $k-1$ has
$\delta-1$ digits equal to $"1"$ in its binary expansion), let $\xi_{1}\left(k\right)=\max\left\{ i\,:\, k_{i}=1\right\} $.
If $\xi_{1}\left(k\right)<n-\mathfrak{h}-1$, then by (\ref{eq:k2k1comp})
we can switch the values of $k_{\xi_{1}\left(k\right)}$ ($=1$) and
$k_{\xi_{1}\left(k\right)+1}$ ($=0$) to obtain a local maximum $k\prime$
such that $g\left(k\right)<g\left(k\prime\right)$. We can repeat
this 'switch' until the final $"1"$ is the $\left\lceil n-\mathfrak{h}-1\right\rceil ^{th}$
term, and all the while obtaining local maxima of $g$, each greater
than the previous (see Remark \ref{remarklargeh} below for the case
$\left\lceil n-\mathfrak{h}-1\right\rceil =0$). Similarly, if $\xi_{1}\left(k\right)\geq\left\lceil n-\mathfrak{h}-1\right\rceil +1$,
let $s_{1}\left(k\right)=\max\left\{ i<\xi_{1}\left(k\right)\,:\, k_{i}=0\right\} $
and let $k\prime$ be the result of switching the terms $k_{s_{1\left(k\right)}}$
($=0$) and $k_{s_{1}\left(k\right)+1}$ ($=1$) in the binary expansion
of $k$. Then again from (\ref{eq:k2k1comp}) it follows that 
\begin{eqnarray}
g\left(k\prime\right)-g\left(k\right) & = & -2^{s_{1}\left(k\right)-1}\left(n-\mathfrak{h}-s_{1}\left(k\right)-1-2\sum_{i=s_{1}\left(k\right)+2}^{n}k_{i}\right)\nonumber \\
 & = & -2^{s_{1}\left(k\right)-1}\left(n-\mathfrak{h}-s_{1}\left(k\right)-1-2\left(\xi_{1}\left(k\right)-s_{1}\left(k\right)-1\right)\right)\nonumber \\
 & = & 2^{s_{1}\left(k\right)-1}\left(2\xi_{1}\left(k\right)-\left(n-\mathfrak{h}-1\right)-s_{1}\left(k\right)-2\right)>0\label{eq:kprimek}
\end{eqnarray}
Thus by switching the values of $k_{s_{1}\left(k\right)}$ and $k_{s_{1}\left(k\right)+1}$,
we obtain a local maximum $k\prime$ which satisfies $g\left(k\prime\right)>g\left(k\right)$.
Applying this repeatedly, we obtain a sequence of integers that are
local maxima with increasing values in $g$, the last of which has
a $"0"$ at the $\xi_{1}\left(k\right)^{th}$, $\xi_{1}\left(k\right)+1^{st},\ldots,n^{th}$
terms in its binary expansion. From these observations we have established
that the value of $\xi_{1}\left(k\right)$ must be equal to $\left\lceil n-\mathfrak{h}-1\right\rceil $
if $k$ is a global maximum. 

We can repeat this process to determine where all other $"1"$s in
the binary expansion of a global maximum must. For $2\leq m\leq\delta$
we can define $\xi_{m}\left(k\right)=\max\left\{ i<\xi_{m-1}\left(k\right)\,:\, k_{i}=1\right\} $
and from (\ref{eq:k2k1comp}) we conclude that if $\xi_{m}\left(k\right)<\left\lceil n-\mathfrak{h}+1-2m\right\rceil $
and $k_{\xi_{m}\left(k\right)+1}=0$, we obtain a greater maximum
by switching $k_{\xi_{m}\left(k\right)+1}$ and $k_{\xi_{m}\left(k\right)}$.
Similarly, if $\xi_{m}\left(k\right)\geq\left\lceil n-\mathfrak{h}+1-2m\right\rceil +1$
then we can define $s_{m}\left(k\right)=\max\left\{ i<\xi_{m}\left(k\right)\,:\, k_{i}=0\right\} $
and give $k$, $k\prime$ analogous definitions to (\ref{eq:kprimek})
to conclude that
\begin{eqnarray}
g\left(k\prime\right)-g\left(k\right) & = & -2^{s_{m}\left(k\right)-1}\left(n-\mathfrak{h}-s_{m}\left(k\right)-1-2\sum_{i=s_{m}\left(k\right)+2}^{n}k_{i}\right)\nonumber \\
 & = & -2^{s_{m}\left(k\right)-1}\left(n-\mathfrak{h}-s_{m}\left(k\right)-1-2\left(\xi_{m}\left(k\right)-s_{m}\left(k\right)-1+m-1\right)\right)\label{eq:kprimek2}\\
 & = & 2^{s_{m}\left(k\right)-1}\left(2\xi_{m}\left(k\right)-\left(n-\mathfrak{h}+1-2m\right)-s_{m}\left(k\right)-2\right)>0\nonumber 
\end{eqnarray}
Thus, applying (\ref{eq:kprimek2}) repeatedly we can obtain a local
maximum of $g$ that has a binary expansion with a $"0"$ at the $\xi_{m}\left(k\right)^{th}$
term and $m-1$ values equal to $"1"$ thereafter. It follows that
if $k$ is a global maximum, $\xi_{m}\left(k\right)=\left\lceil n-\mathfrak{h}+1-2m\right\rceil $.
Note that for $m=\delta$, $\left\lceil n-\mathfrak{h}+1-2m\right\rceil \in\left\{ 0,1\right\} $
and hence we set $\xi_{\delta}=1$ which agrees with our previous
observation that all local maxima are odd. Therefore, for $\mathfrak{h}<n-1$
(see Remark \ref{remarklargeh}) the maximum of $g$ is attained at 

\begin{eqnarray}
k^{\star} & = & 2^{\xi_{1}-1}+2^{\xi_{2}-1}+\ldots+2^{\xi_{\delta-1}-1}+1\label{eq:kstar}\\
 & = & 2^{\left\lceil n-\mathfrak{h}-2\right\rceil }+2^{\left\lceil n-\mathfrak{h}-4\right\rceil }+\ldots+2^{\left\lceil n-\mathfrak{h}-2\delta+2\right\rceil }+1\nonumber 
\end{eqnarray}
Following the derivations in Lemma \ref{lem:2torsum} and Lemma \ref{lem:abdiff}
\begin{eqnarray*}
\sum_{i=1}^{k^{\star}-1}q\left(i\right) & = & q\left(k^{\star}-1\right)+\sum_{i=1}^{k^{\star}-2}q\left(i\right)\\
 & = & \left(\delta-1\right)+\sum_{m=1}^{\delta-1}\sum_{i=1}^{2^{\left(\left\lceil n-\mathfrak{h}-2m\right\rceil \right)}-1}q\left(i\right)+\sum_{m=1}^{\delta-1}\left(m-1\right)2^{\left\lceil n-\mathfrak{h}-2m\right\rceil }\\
 & = & \left(\delta-1\right)+\sum_{m=1}^{\delta-1}\left(\left(\left\lceil n-\mathfrak{h}\right\rceil -2m\right)2^{\left\lceil n-\mathfrak{h}-2m\right\rceil -1}+\left(2m-2\right)2^{\left\lceil n-\mathfrak{h}-2m\right\rceil -1}\right)\\
 & = & \left(\delta-1\right)+\sum_{m=1}^{\delta-1}2^{\left\lceil n-\mathfrak{h}-2m\right\rceil -1}\left(\left\lceil n-\mathfrak{h}\right\rceil -2\right)
\end{eqnarray*}
and thus 
\begin{eqnarray*}
g\left(k^{\star}\right) & = & \left(1+\sum_{m=1}^{\delta-1}2^{\left\lceil n-\mathfrak{h}-2m\right\rceil }\right)\left(n-\mathfrak{h}\right)-2\left(\delta-1\right)-\left(\left\lceil n-\mathfrak{h}\right\rceil -2\right)\sum_{m=1}^{\delta-1}2^{\left\lceil n-\mathfrak{h}-2m\right\rceil }\\
 & = & \left(n-\mathfrak{h}-2\delta+2\right)+\left(n-\mathfrak{h}-\left\lceil n-\mathfrak{h}\right\rceil +2\right)\sum_{m=1}^{\delta-1}2^{\left\lceil n-\mathfrak{h}-2m\right\rceil }\\
 & = & \left(n-\mathfrak{h}-2\delta+2\right)+2^{\left\lceil n-\mathfrak{h}-2\delta+2\right\rceil }\left(n-\mathfrak{h}-\left\lceil n-\mathfrak{h}\right\rceil +2\right)\left(4^{\delta-1}-1\right)/3
\end{eqnarray*}
Finally, note that $g\left(k^{\star}\right)=\frac{1}{3}\left(2-\mathfrak{h}+\left\lfloor \mathfrak{h}\right\rfloor \right)\left(2^{2\delta-1}-2\right)-1$
when $\left\lfloor n-\mathfrak{h}\right\rfloor $ is even, and $g\left(k^{\star}\right)=\frac{1}{3}\left(2-\mathfrak{h}+\left\lfloor \mathfrak{h}\right\rfloor \right)\left(2^{2\delta}-4\right)$
when $\left\lfloor n-\mathfrak{h}\right\rfloor $ is odd.\end{proof}
\begin{rem}
\label{remarklargeh}The above derivation made an implicit assumption
that $\left\lceil n-\mathfrak{h}-1\right\rceil \geq1$. Note that
if $\left\lceil n-\mathfrak{h}-1\right\rceil =0$, then $\delta=1$
and it is immediate from (\ref{eq:k2k1comp}) that the only $"1"$
in the binary expansion of $k$ belongs to $k_{1}$. Therefore, in
this special case $k^{\star}=1$ and $\Gamma^{\star}=n-\mathfrak{h}$
are the solutions to the above problem.
\end{rem}

\section{\label{sec:Critical-and-protocritical}Critical and protocritical
sets}

In this section we will determine properties of configurations in
$\mathscr{P}^{\star}$ and $\mathscr{C}^{\star}$ that are relevant
to the results in Section \ref{sec:intro}. In particular, these will
be used to obtain an expression for the prefactor $K$ in Theorem
\ref{thm:expcrossovertime}. We will begin by introducing a variational
equation that gives us an expression for $K$, derived in Lemma 16.17
in \cite{key-2}, and in the case of our model equivalent to 
\begin{equation}
1/K=\min_{C_{1},\ldots,C_{I}}\min_{h:\, S^{\star}\rightarrow\left[0,1\right],\, h|_{S_{\boxminus}}=1,\, h|_{S_{\boxplus}}=0,\, h|_{S_{i}=C_{i}}}\frac{1}{2}\sum_{\xi,\xi\prime\in S^{\star}}\mathbf{1}_{\left\{ \xi\sim\xi\prime\right\} }\left[h\left(\xi\right)-h\left(\xi\prime\right)\right]^{2}\label{eq:variationalform}
\end{equation}
Here the sequence $\left\{ S_{i}\right\} _{i=1}^{I}$ are sets $S_{i}\subseteq\Omega$
that are mutually disjoint and satisfy
\begin{equation}
\sigma\in S_{i}\mbox{ if and only if }\mathcal{H}\left(\sigma\right)<\mathcal{H}\left(\gamma_{k^{\star}}\right)\mbox{ and }\Phi\left(\sigma,\boxminus\right)=\Phi\left(\sigma,\boxplus\right)=\mathcal{H}\left(\gamma_{k^{\star}}\right)\label{eq:wells}
\end{equation}
The terms $C_{1},\ldots,C_{I}$ are real numbers corresponding to
the values that $h$ takes on $S_{1},\ldots,S_{I}$. The set $S_{\boxminus}$
is defined by 
\[
S_{\boxminus}=\left\{ \sigma\in\Omega:\,\Phi\left(\sigma,\boxminus\right)<H\left(\gamma_{k^{\star}}\right)\right\} 
\]
and a similar definition is given to $S_{\boxplus}$. Lastly, $S^{\star}\subseteq\Omega$
is the set of all $\sigma\in\Omega$ such that $\Phi\left(\sigma,\boxminus\right)\leq\mathcal{H}\left(\gamma_{k^{\star}}\right)$
(and hence also $\Phi\left(\sigma,\boxplus\right)\leq\mathcal{H}\left(\gamma_{k^{\star}}\right)$).
Our aim now is to evaluate the right-hand side of (\ref{eq:variationalform})
by first showing that it can be simplified considerably.

\medskip{}

Recall from equation (\ref{eq:kstar}) that 
\begin{equation}
\Upsilon_{k^{\star}}=\left\{ v=\left(a_{1},\ldots,a_{n}\right)\in Q_{n}\,\vert\,\sum_{i=1}^{n}a_{i}2^{i-1}<k^{\star}\right\} \label{eq:sstarkstar}
\end{equation}
is where $\mathcal{H}$ attains its unique maximum along the optimal
path $\gamma$ defined in (\ref{eq:optpath}). We claim that $\Upsilon_{k^{\star}-1}$
and $\Upsilon_{k^{\star}}$ are in the protocritical set $\mathscr{P}^{\star}$
and critical set $\mathscr{C}^{\star}$, respectively. Indeed, the
first condition in (\ref{eq:critandprot}) is satisfied since $\left|\Upsilon_{k^{\star}-1}\triangle\Upsilon_{k^{\star}}\right|=1$.
The second condition is also immediate, since $\left(\gamma_{1},\ldots,\gamma_{k^{\star}-1}\right)$
is a path from $\boxminus$ to $\Upsilon_{k^{\star}-1}$, and for
$1\leq i\leq k^{\star}-1$ 
\[
\mathcal{H}\left(\gamma_{i}\right)<\mathcal{H}\left(\gamma_{k^{\star}}\right)=\Phi\left(\Upsilon_{k^{\star}-1},\boxplus\right)
\]
since any path from $\Upsilon_{k^{\star}-1}$ to $\boxplus$ must
pass through some configuration of size $k^{\star}$. Thus $\Phi\left(\Upsilon_{k^{\star}-1},\boxminus\right)<\Phi\left(\Upsilon_{k^{\star}-1},\boxplus\right)$.
The third condition is also easy to verify: since $\mathcal{H}$ attains
its maximum along $\gamma$ at $\gamma_{k^{\star}}$($=\Upsilon_{k^{\star}}$),
the path $\left(\gamma_{k^{\star}},\gamma_{k^{\star}+1},\ldots,\gamma_{n}\right)$
from $\gamma_{k^{\star}}$ to $\boxplus$ satisfies 
\[
\mathcal{H}\left(\gamma_{k^{\star}+i}\right)-\mathcal{H}\left(\boxminus\right)\leq\Gamma^{\star}\mbox{ for all }0\leq i\leq n-k^{\star}
\]
and 
\begin{equation}
\mathcal{H}\left(\gamma_{k^{\star}}\right)=\Phi\left(\gamma_{k^{\star}+i},\boxminus\right)>\Phi\left(\gamma_{k^{\star}+i},\boxplus\right)\mbox{ for all }1\leq i\leq n-k^{\star}\label{eq:condsatisfied}
\end{equation}
The equality in (\ref{eq:condsatisfied}) also uses the fact that
any path from $\gamma_{k^{\star}+i}$ to $\boxminus$ must pass through
some configuration of size $k^{\star}$, and every configuration of
size $k^{\star}$ has energy greater than or equal to $\mathcal{H}\left(\gamma_{k^{\star}}\right)$.
The inequality follows from the fact that $\mathcal{H}$ has a unique
maximum along the path $\gamma$, attained at $\gamma_{k^{\star}}$.

If $\varphi$ is any isomorphism of $\mathcal{Q}_{n}$, then the configurations
$\varphi\left(\Upsilon_{k^{\star}-1}\right)$ and $\varphi\left(\Upsilon_{k^{\star}}\right)$
also satisfy the requirements in (\ref{eq:critandprot}) and are in
$\mathscr{P}^{\star}$ and $\mathscr{C}^{\star}$, respectively. Furthermore,
from Lemma \ref{lem:allminimalsets} it follows that if $\sigma\in\Omega$
with $\left|\sigma\right|=k^{\star}$ and $\sigma\neq\varphi\left(\Upsilon_{k^{\star}}\right)$
for any isomorphism $\varphi$, then $\Gamma^{\star}=\mathcal{H}\left(\Upsilon_{k^{\star}}\right)-\mathcal{H}\left(\boxminus\right)<\mathcal{H}\left(\sigma\right)-\mathcal{H}\left(\boxminus\right)$
and hence $\sigma\notin\mathscr{C}^{\star}$. Thus we conclude that
\begin{eqnarray}
\mathscr{C}^{\star} & = & \left\{ \varphi\left(S_{k^{\star}}^{\star}\right)\,:\,\varphi\,\mbox{ is an isomorphism of }\mathcal{Q}_{n}\right\} \label{eq:cstar}\\
\mathscr{P}^{\star} & \subseteq & \left\{ \varphi\left(S_{k^{\star}-1}^{\star}\right)\,:\,\varphi\,\mbox{ is an isomorphism of }\mathcal{Q}_{n}\right\} \nonumber 
\end{eqnarray}
Furthermore, 
\begin{lem}
\label{lem:nootherwells}There is no configurations $\sigma\in\Omega$
that satisfies (\ref{eq:wells}). Hence the index $I$ in equation
(\ref{eq:variationalform}) satisfies $I=0$. \end{lem}
\begin{proof}
This is the only result where we make use of $\mathfrak{h}\neq\frac{a}{b}$
for any $a\in\mathbb{N}$ and $b\in\left\{ 1,\ldots,2^{n}\right\} $.
By (\ref{eq:hamiltonian}) with $\mathfrak{J}=1$, this restriction
on $\mathfrak{h}$ implies
\begin{equation}
\forall\sigma_{1},\sigma_{2}\in\Omega,\,\left|\sigma_{1}\right|\neq\left|\sigma_{2}\right|\Rightarrow\mathcal{H}\left(\sigma_{1}\right)\neq\mathcal{H}\left(\sigma_{2}\right)\label{eq:irratconsequence}
\end{equation}
Let $\sigma\in\Omega$ be such that $\Phi\left(\sigma,\boxminus\right)\leq\mathcal{H}\left(\gamma_{k^{\star}}\right)$,
$\Phi\left(\sigma,\boxplus\right)\leq\mathcal{H}\left(\gamma_{k^{\star}}\right)$
and $\mathcal{H}\left(\sigma\right)<\mathcal{H}\left(\gamma_{k^{\star}}\right)$,
and suppose first that $\left|\sigma\right|>k^{\star}$. Then by (\ref{eq:irratconsequence})
there is some $\zeta\in\mathscr{C}^{\star}$ and a path $\sigma=\sigma_{0},\ldots,\sigma_{m}=\zeta$
such that $\mathcal{H}\left(\sigma_{i}\right)<\mathcal{H}\left(\gamma_{k^{\star}}\right)=\mathcal{H}\left(\zeta\right)$
for all $0\leq i\leq m-1$. Observe that $\sigma_{m-1}=\zeta\cup\left\{ w\right\} $
for some $w\notin\zeta$. Let us also take a uniformly optimal path
$\zeta=\zeta_{0},\ldots\zeta_{2^{n}-k^{\star}}=\boxplus$, similar
to a segment of the path $\gamma$ in (\ref{eq:optpath}). If $w\notin\zeta_{i}$,
\begin{eqnarray*}
\mathcal{H}\left(\zeta_{i}\right)-\mathcal{H}\left(\sigma_{m-1}\cup\zeta_{i}\right) & = & \mathcal{H}\left(\zeta_{i}\right)-\mathcal{H}\left(\left\{ w\right\} \cup\zeta_{i}\right)\\
 & = & \left|E\left(\zeta_{i},\overline{\zeta_{i}}\right)\right|-\left(\left|E\left(\zeta_{i},\overline{\zeta_{i}}\right)\right|+\left|E\left(\left\{ w\right\} ,\overline{\left\{ w\right\} }\right)\right|-2\left|E\left(\left\{ w\right\} ,\zeta_{i}\right)\right|\right)+\mathfrak{h}\\
 & = & 2\left|E\left(\left\{ w\right\} ,\zeta_{i}\right)\right|-n+\mathfrak{h}\\
 & \geq & 2\left|E\left(\left\{ w\right\} ,\zeta\right)\right|-n+\mathfrak{h}\\
 & = & \mathcal{H}\left(\zeta\right)-\mathcal{H}\left(\left\{ w\right\} \cup\zeta\right)=\mathcal{H}\left(\zeta\right)-\mathcal{H}\left(\sigma_{m-1}\right)>0
\end{eqnarray*}
And if $w\in\zeta_{i}$ for some $i\geq1$, then $\sigma_{m-1}\cup\zeta_{i}=\zeta_{i}$
and it follows that $\mathcal{H}\left(\zeta_{i}\right)<\mathcal{H}\left(\zeta\right)$
since $\mathcal{H}$ has a unique maximum at $\zeta$ along this path.
This shows that on the path $\left(\sigma_{0},\ldots,\sigma_{m-1},\sigma_{m-1}\cup\zeta_{1},\ldots,\sigma_{m-1}\cup\zeta_{2^{n}-k^{\star}}\right)$
from $\sigma$ to $\boxplus$, $\mathcal{H}$ is strictly less than
$\mathcal{H}\left(\gamma_{k^{\star}}\right)$. Thus $\Phi\left(\sigma,\boxplus\right)<\mathcal{H}\left(\zeta\right)=\mathcal{H}\left(\gamma_{k^{\star}}\right)$. 

Similarly, if $\left|\sigma\right|<k^{\star}$ and $\sigma=\sigma_{0},\ldots,\sigma_{m}=\zeta$
is a path from $\sigma$ to some $\zeta\in\mathscr{C}^{\star}$ such
that $\mathcal{H}\left(\sigma_{i}\right)<\mathcal{H}\left(\gamma_{k^{\star}}\right)$
for $0\leq i<m$, then $\sigma_{m-1}=\zeta\backslash\left\{ w\right\} $
for some $w\in\zeta$, and if $\zeta=\zeta_{0},\ldots\zeta_{k^{\star}}=\boxminus$
is a uniformly optimal path and $w\in\zeta_{i}$, 
\begin{eqnarray*}
\mathcal{H}\left(\zeta_{i}\right)-\mathcal{H}\left(\sigma_{m-1}\cap\zeta_{i}\right) & = & \mathcal{H}\left(\zeta_{i}\right)-\mathcal{H}\left(\zeta_{i}\backslash\left\{ w\right\} \right)\\
 & = & \left|E\left(\zeta_{i},\overline{\zeta_{i}}\right)\right|-\left(\left|E\left(\zeta_{i},\overline{\zeta_{i}}\right)\right|+\left|E\left(\left\{ w\right\} ,\zeta_{i}\right)\right|-\left(n-\left|E\left(\left\{ w\right\} ,\zeta_{i}\right)\right|\right)\right)-\mathfrak{h}\\
 & = & n-2\left|E\left(\left\{ w\right\} ,\zeta_{i}\right)\right|-\mathfrak{h}\\
 & \geq & n-2\left|E\left(\left\{ w\right\} ,\zeta\right)\right|-\mathfrak{h}\\
 & = & \mathcal{H}\left(\zeta\right)-\mathcal{H}\left(\zeta\backslash\left\{ w\right\} \right)=\mathcal{H}\left(\zeta\right)-\mathcal{H}\left(\sigma_{m-1}\right)>0
\end{eqnarray*}
And if $w\notin\zeta_{i}$, then $\sigma_{m-1}\cap\zeta_{i}=\zeta_{i}$
and by the unique maximum of the path, $\mathcal{H}\left(\zeta_{i}\right)<\mathcal{H}\left(\zeta\right)$.
Hence this time we have that $\Phi\left(\sigma,\boxminus\right)<\mathcal{H}\left(\gamma_{k^{\star}}\right)$. 
\end{proof}
\medskip{}
Observe that for any distinct $C_{1},C_{2}\in\mathscr{C}^{\star}$,
$\left|C_{1}\triangle C_{2}\right|>1$ and hence $\mathcal{E}_{n}\cap\left(\mathscr{C}^{\star}\times\mathscr{C}^{\star}\right)=\emptyset$.
As a consequence of this and of Lemma \ref{lem:nootherwells}, equation
(\ref{eq:variationalform}) simplifies to 
\begin{eqnarray}
1/K & = & \min_{h:\,\mathscr{C}^{\star}\rightarrow\left[0,1\right]}\sum_{\sigma\in\mathscr{C}^{\star}}\left[1-h\left(\sigma\right)\right]^{2}N^{-}\left(\sigma\right)+\left[h\left(\sigma\right)\right]^{2}N^{+}\left(\sigma\right)\nonumber \\
 & = & \sum_{\sigma\in\mathscr{C}^{\star}}\frac{N^{-}\left(\sigma\right)N^{+}\left(\sigma\right)}{N^{-}\left(\sigma\right)+N^{+}\left(\sigma\right)}=\left|\mathscr{C}^{\star}\right|\frac{N^{-}\left(\sigma\right)N^{+}\left(\sigma\right)}{N^{-}\left(\sigma\right)+N^{+}\left(\sigma\right)}\label{eq:variationalformsimplified}
\end{eqnarray}
where $\sigma$ is any configuration in $\mathscr{C}^{\star}$ and
\begin{eqnarray}
N^{-}\left(\sigma\right) & = & \left|\left\{ \sigma\prime\in\mathscr{P}^{\star}:\,\sigma\sim\sigma\prime\right\} \right|\nonumber \\
N^{+}\left(\sigma\right) & = & \left|\left\{ \sigma\prime\in\mathscr{B}^{\star}:\,\sigma\sim\sigma\prime\right\} \right|\label{eq:nplusminus}
\end{eqnarray}
The second line in the equality follows from the substitution 
\[
h\left(\sigma\right)=\mathbb{P}_{\sigma}\left(\tau_{S_{\boxminus}}<\tau_{S_{\boxplus}}\right)=\frac{N^{-}\left(\sigma\right)}{N^{-}\left(\sigma\right)+N^{+}\left(\sigma\right)}
\]
which is a solution to this variational problem (see for example equation
(16.2.4) in \cite{key-2}). By symmetry of the hypercube, $N^{-}$
and $N^{+}$ are constant on $\mathscr{C}^{\star}$, which justifies
the last equality in (\ref{eq:variationalformsimplified}). Our final
task is to determine the size of the set $\mathscr{C}^{\star}$ and
the values of $N^{-}\left(\sigma\right)$ and $N^{+}\left(\sigma\right)$. 

\medskip{}

For a vertex $v\in V_{n}$ and $1\leq s\leq n$, let $\theta_{s}\left(v\right)\in V_{n}$
be the vertex that agrees with $v$ at every co-ordinate except at
$v\left(s\right)$. If $\mathcal{Q}_{r}$ is an $r$-dimensional sub-cube
of $\mathcal{Q}_{n}$ ($r<n$), and $1\leq s\leq n$ is such that
$v\left(s\right)=w\left(s\right)$ for every $v,w\in\mathcal{Q}_{r}$
(in other words, the co-ordinate $s$ lies outside $\mathcal{Q}_{r}$),
define $\theta_{s}\left(\mathcal{Q}_{r}\right)$ by 
\begin{equation}
\theta_{s}\left(\mathcal{Q}_{r}\right):=\left\{ \theta_{s}\left(v\right):\, v\in\mathcal{Q}_{r}\right\} \label{eq:thetaQr}
\end{equation}
Note that $\theta_{s}\left(\mathcal{Q}_{r}\right)$ is also an $r$-dimensional
sub-cube of $\mathcal{Q}_{n}$. We will also say in this case that
$s$ is an \emph{external co-ordinate }of the sub-cube $\mathcal{Q}_{r}$. 

Now by Remark \ref{remgoodset}, every configuration in $\mathscr{C}^{\star}$
can also be constructed as follows. Start with any $\left\lceil n-\mathfrak{h}-2\right\rceil $-dimensional
sub-cube $\mathcal{Q}_{1}$. There are ${n \choose \left\lceil n-\mathfrak{h}-2\right\rceil }\times2^{n-\left\lceil n-\mathfrak{h}-2\right\rceil }$
different choices for such a sub-cube. Let $s_{1}$ be any external
co-ordinate of $\mathcal{Q}_{1}$, and let $\mathcal{Q}_{2}$ be a
$\left\lceil n-\mathfrak{h}-4\right\rceil $-dimensional sub-cube
of $\theta_{s_{1}}\left(\mathcal{Q}_{1}\right)$. There are $\left(n-\left\lceil n-\mathfrak{h}-2\right\rceil \right)\times{\left\lceil n-\mathfrak{h}-2\right\rceil  \choose \left\lceil n-\mathfrak{h}-4\right\rceil }\times2^{2}$
ways to go about selecting $\mathcal{Q}_{2}$. Equation (\ref{eq:kstar})
implies that we should continue with this construction until we have
chosen a $\left\lceil n-\mathfrak{h}-2\delta+2\right\rceil $-dimensional
sub-cube $\mathcal{Q}_{\delta-1}$ followed by a single vertex from
the sub-cube $\theta_{s_{\delta-1}}\left(\mathcal{Q}_{s_{\delta-1}}\right)$,
which will be identified with the $0$-dimensional sub-cube $\mathcal{Q}_{\delta}$.
For $i\geq2$, there are always two choices for the external co-ordinate
$s_{i}$ of $\mathcal{Q}_{i}$, since both $\mathcal{Q}_{i}$ and
$\theta_{s_{i}}\left(\mathcal{Q}_{i}\right)$ lie inside $\theta_{s_{i-1}}\left(\mathcal{Q}_{i-1}\right)$
(see Figure \ref{fig:pstar}). And there are ${\left\lceil n-\mathfrak{h}-2i\right\rceil  \choose \left\lceil n-\mathfrak{h}-2i-2\right\rceil }$
ways to choose the co-ordinates of $\mathcal{Q}_{i+1}$, and $2^{2}$
ways to fix the two external co-ordinates of $\mathcal{Q}_{i+1}$(for
$i+1<\delta$) that are in $\theta_{s_{i}}\left(\mathcal{Q}_{i}\right)$
. Therefore, $\left|\mathscr{C}^{\star}\right|$ is given by 
\begin{eqnarray}
\left|\mathscr{C}^{\star}\right| & = & {n \choose \left\lceil n-\mathfrak{h}-2\right\rceil }\times2^{n-\left\lceil n-\mathfrak{h}-2\right\rceil }\times\left(n-\left\lceil n-\mathfrak{h}-2\right\rceil \right)\times{\left\lceil n-\mathfrak{h}-2\right\rceil  \choose \left\lceil n-\mathfrak{h}-4\right\rceil }\times2^{2}\label{eq:sizecstar}\\
 &  & \times\left[\prod_{i=2}^{\delta-2}2\times{\left\lceil n-\mathfrak{h}-2i\right\rceil  \choose \left\lceil n-\mathfrak{h}-2i-2\right\rceil }\times2^{2}\right]\times2\times2^{\left\lceil n-\mathfrak{h}-2\delta+2\right\rceil }\nonumber \\
 & = & 2^{3\left(\delta-2\right)+n-\left\lceil n-\mathfrak{h}-2\right\rceil +\left\lceil n-\mathfrak{h}-2\delta+2\right\rceil }{n \choose \left\lceil n-\mathfrak{h}-2\right\rceil }\left(n-\left\lceil n-\mathfrak{h}-2\right\rceil \right)\left[\prod_{i=1}^{\delta-2}{\left\lceil n-\mathfrak{h}-2i\right\rceil  \choose \left\lceil n-\mathfrak{h}-2i-2\right\rceil }\right]\nonumber \\
 & = & \frac{n!2^{2\left(\delta-2\right)+n-\left\lceil n-\mathfrak{h}-2\right\rceil +\left\lceil n-\mathfrak{h}-2\delta+2\right\rceil }}{\left(n-\left\lceil n-\mathfrak{h}-2\right\rceil -1\right)!\left\lceil n-\mathfrak{h}-2\delta+2\right\rceil }=\frac{n!2^{n-4}}{\left(n-\left\lceil n-\mathfrak{h}-2\right\rceil -1\right)!\left\lceil n-\mathfrak{h}-2\delta+2\right\rceil }\nonumber 
\end{eqnarray}

\medskip{}

We can also use the above construction of configurations in $\mathscr{C}^{\star}$
to get a complete representation of the set $\mathscr{P}^{\star}$
(note that (\ref{eq:cstar}) gives only a subset of $\mathscr{P}^{\star}$).
Suppose that $v\in\Upsilon_{k^{\star}}$ belongs to the sub-cube $\mathcal{Q}_{i}$
for some $1\leq i\leq\delta-1$, as defined in the preceding paragraph.
Then $v$ has $\left\lceil n-\mathfrak{h}-2i\right\rceil $ neighbours
in $\mathcal{Q}_{i}$, one neighbour in each of $\mathcal{Q}_{1},\ldots,\mathcal{Q}_{i-1}$,
and one or zero neighbours in $\mathcal{Q}_{i+1}$ (see Figure \ref{fig:pstar}).
\begin{figure}
\caption{\label{fig:pstar}Schematic representation of a configuration in $\mathscr{C}^{\star}$.
Only the largest four sub-cubes are shown. The two vertices $w_{1},w_{2}\in\mathcal{Q}_{3}$
have zero and one neighbour in $\mathcal{Q}_{4}$, respectively. }

\begin{tikzpicture}

\draw[gray, very thin] (12.7,0.3) node {\tiny {$\theta^{-1}_{s_{3}}\left(\mathcal{Q}_{4}\right)$}} (12,0) rectangle  (13,1);

\draw[thick] (4.5,4.5) node {$\mathcal{Q}_{1}$} (0,0) rectangle  (8,8);
\draw[thick] (10.5,2.5) node {$\mathcal{Q}_{2}$} (8,0) rectangle (12,4); 
\draw[thick] (13.5,1.5) node {$\mathcal{Q}_{3}$} (12,0) rectangle (14,2); 
\draw[thick]  (14.5,0.5) node {$\mathcal{Q}_{4}$} (14,0) rectangle (15,1); 

\draw[thick]  (13.2,1.7) node [blue] {\tiny{\textbullet $w_{1}$}};
\draw[thick]  (12.5,0.8) node [blue] {\tiny{\textbullet $w_{2}$}};

\draw[thin, blue!50]  (14.9,0.8) node {\tiny{\textbullet $\theta_{s_{3}}\left( w_{2}\right)$}};
\draw[thin, blue!50]  (8.9,0.8) node {\tiny{\textbullet $\theta^{-1}_{s_{2}}\left( w_{2}\right)$}};
\draw[thin, blue!50]  (4.9,0.8) node {\tiny{\textbullet $\theta^{-1}_{s_{1}}\left( w_{2}\right)$}};
\draw[dashed, blue!50] (12.42,0.8)  to[out=10,in=170] (14.45,0.8);
\draw[dashed, blue!50] (12.42,0.8) to[out=170,in=10] (8.42,0.8);
\draw[dashed, blue!50] (12.42,0.8) to[out=170,in=10] (4.42,0.8);

\draw[thin, blue!50]  (9.7,1.7) node {\tiny{\textbullet $\theta^{-1}_{s_{2}}\left( w_{2}\right)$}};
\draw[thin, blue!50]  (5.7,1.7) node {\tiny{\textbullet $\theta^{-1}_{s_{1}}\left( w_{2}\right)$}};
\draw[dashed, blue!50] (13.12,1.7) to[out=170,in=10] (9.2,1.7);
\draw[dashed, blue!50] (13.12,1.7) to[out=170,in=10] (5.2,1.7);

\draw[gray, thin] (12,4.5) node  {\tiny {$\theta_{s_{1}}\left(\mathcal{Q}_{1}\right)$}} ;
\draw[gray, thin] (14,2.5) node  {\tiny {$\theta_{s_{2}}\left(\mathcal{Q}_{2}\right)$}} ;
\draw[gray, thin] (15,1.5) node  {\tiny {$\theta_{s_{1}}\left(\mathcal{Q}_{1}\right)$}} ;

\draw[dashed, thin]  (8,0) rectangle (16,8); 
\draw[dashed, thin]  (12,4) -- (16,4); 
\draw[dashed, thin]  (14,2) -- (16,2); 
\draw[dashed, thin]  (14,1) -- (16,1);

\end{tikzpicture}
\end{figure}
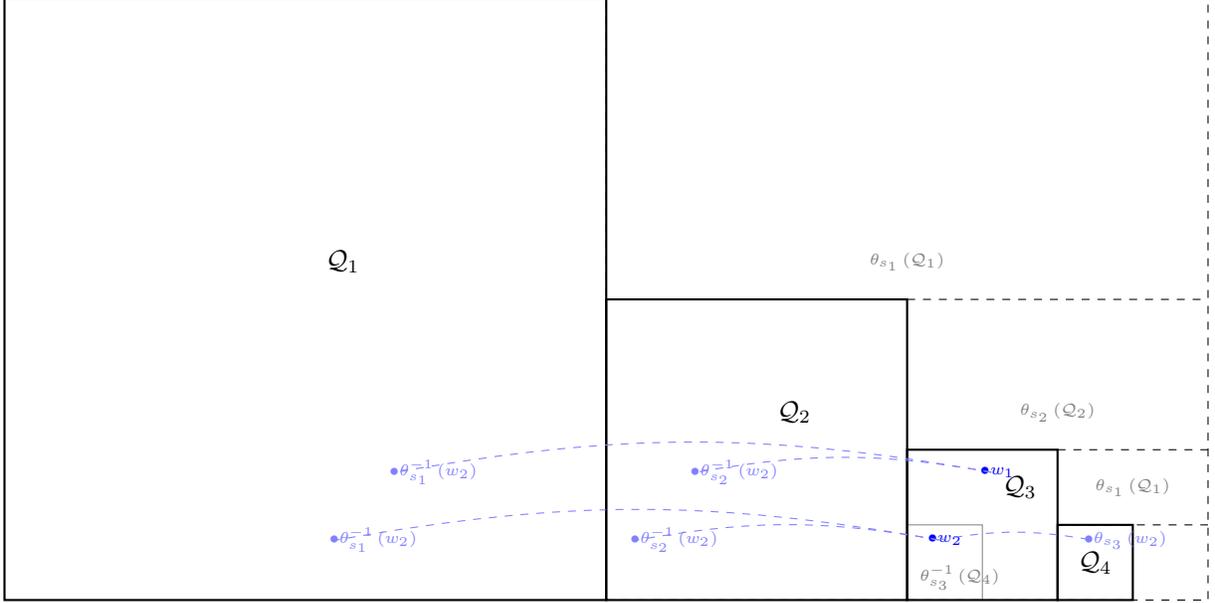
Similarly $v\in\mathcal{Q}_{\delta}$ has zero neighbours in $\mathcal{Q}_{\delta}$
and one in each of $\mathcal{Q}_{1},\ldots,\mathcal{Q}_{\delta-1}$.
Thus if $1\leq i\leq\delta-1$, $\left|E\left(v,\Upsilon_{k^{\star}}\right)\right|=\left\lceil n-\mathfrak{h}-2i\right\rceil +i-1=\left\lceil n-\mathfrak{h}-i-1\right\rceil $
if $v$ has no neighbours in $\mathcal{Q}_{i+1}$, $\left|E\left(v,\Upsilon_{k^{\star}}\right)\right|=\left\lceil n-\mathfrak{h}-2i\right\rceil +i=\left\lceil n-\mathfrak{h}-i\right\rceil $
if $v$ has one neighbour in $\mathcal{Q}_{i+1}$, and $\left|E\left(v,\Upsilon_{k^{\star}}\right)\right|=\delta-1$
if $v\in\mathcal{Q}_{\delta}$. Hence, for $v\in\mathcal{Q}_{i}$
and $1\leq i\leq\delta-1$, 
\begin{equation}
\mathcal{H}\left(\Upsilon_{k^{\star}}\backslash\left\{ v\right\} \right)=\mathcal{H}\left(\Upsilon_{k^{\star}}^{\star}\right)+\left\lceil n-\mathfrak{h}-i-1\right\rceil -\left(n-\left\lceil n-\mathfrak{h}-i-1\right\rceil \right)+\mathfrak{h}=\mathcal{H}\left(\Upsilon_{k^{\star}}^{\star}\right)+2\left\lceil n-\mathfrak{h}-i-1\right\rceil -n+\mathfrak{h}\label{eq:critminusone}
\end{equation}
if $v$ with no neighbours in $\mathcal{Q}_{i+1}$, and 
\begin{equation}
\mathcal{H}\left(\Upsilon_{k^{\star}}\backslash\left\{ v\right\} \right)=\mathcal{H}\left(\Upsilon_{k^{\star}}\right)+2\left\lceil n-\mathfrak{h}-i\right\rceil -n+\mathfrak{h}\label{eq:critminusone2}
\end{equation}
if $v$ has a neighbour in $\mathcal{Q}_{i+1}$. If $v\in\mathcal{Q}_{\delta}$,
\begin{equation}
\mathcal{H}\left(\Upsilon_{k^{\star}}\backslash\left\{ v\right\} \right)=\mathcal{H}\left(\Upsilon_{k^{\star}}\right)+2\left(\delta-1\right)-n+\mathfrak{h}\label{eq:critminusone3}
\end{equation}
Note that for (\ref{eq:critminusone3}), $\mathcal{H}\left(\Upsilon_{k^{\star}}\backslash\left\{ v\right\} \right)<\mathcal{H}\left(\Upsilon_{k^{\star}}\right)$
if and only if $\left(n-\mathfrak{h}\right)/2>\delta-1=\left\lceil \left(n-\mathfrak{h}\right)/2\right\rceil -1$,
which is always true. Furthermore, 
\begin{equation}
2\left\lceil n-\mathfrak{h}-j\right\rceil -n+\mathfrak{h}<0\mbox{ if and only if }j\geq\left\lfloor \left\lceil n-\mathfrak{h}\right\rceil /2\right\rfloor +1\label{eq:least j}
\end{equation}
which does not hold if $j\leq\delta-1=\left\lceil \left(n-\mathfrak{h}\right)/2\right\rceil -1$.
Hence (\ref{eq:critminusone}) and (\ref{eq:critminusone2}) never
satisfy $\mathcal{H}\left(\Upsilon_{k^{\star}}\backslash\left\{ v\right\} \right)<\mathcal{H}\left(\Upsilon_{k^{\star}}\right)$,
and in particular this implies that $\mathcal{H}\left(\Upsilon_{k^{\star}}\backslash\left\{ v\right\} \right)<\mathcal{H}\left(\Upsilon_{k^{\star}}\right)$
if and only if $\Upsilon_{k^{\star}}\backslash\left\{ v\right\} =\Upsilon_{k^{\star}-1}$.
This immediately gives
\begin{lem}
\label{lem:pstarexact} Using the above notation,
\[
\mathscr{P}^{\star}=\left\{ \varphi\left(\Upsilon_{k^{\star}-1}\right):\,\varphi\,\mbox{is an isomorphism of }\mathcal{Q}_{n}\right\} 
\]
and 
\[
N^{-}\left(\sigma\right)=\left|\left\{ \sigma\prime\in\mathscr{P}^{\star}:\,\sigma\sim\sigma\prime\right\} \right|=1
\]

\medskip{}

\end{lem}
\medskip{}

Note that by Lemma \ref{lem:pstarexact}, hypothesis $(H2)$ is now
also verified. Let us now also define 
\begin{equation}
\mathscr{B}^{\star}:=\left\{ \sigma\in S_{\boxplus}:\,\sigma\sim\sigma\prime\mbox{ for some }\sigma\prime\in\mathscr{C}^{\star}\right\} \label{eq:bstar}
\end{equation}
We proceed with investigating the configurations $\sigma\in\mathscr{B}^{\star}$,
in order to obtain an expression for $N^{+}\left(\sigma\right)$.
For $w\notin\Upsilon_{k^{\star}}^{\star}$, 
\[
\mathcal{H}\left(\Upsilon_{k^{\star}}\cup\left\{ w\right\} \right)=\mathcal{H}\left(\Upsilon_{k^{\star}}\right)-\left|E\left(\left\{ w\right\} ,\Upsilon_{k^{\star}}\right)\right|+\left(n-\left|E\left(\left\{ w\right\} ,\Upsilon_{k^{\star}}\right)\right|\right)-\mathfrak{h}
\]
and this is less than $\mathcal{H}\left(\Upsilon_{k^{\star}}\right)$
if and only if 
\begin{equation}
\frac{n-\mathfrak{h}}{2}<\left|E\left(\left\{ w\right\} ,\Upsilon_{k^{\star}}\right)\right|\label{eq:winbstar}
\end{equation}
Observe that if $w\notin\theta_{u}\left(\mathcal{Q}_{i}\right)$ for
any $1\leq i\leq\delta$ and any external co-ordinate $u$ of $\mathcal{Q}_{i}$,
then $\left|E\left(\left\{ w\right\} ,\Upsilon_{k^{\star}}\right)\right|=0$.
Now for $w\in\overline{\Upsilon_{k^{\star}}}$, let $\Xi\left(w\right):=\min\left\{ j\geq1:\, w\notin\theta_{s_{j}}\left(\mathcal{Q}_{j}\right)\right\} $.
Then $\left|E\left(\left\{ w\right\} ,\Upsilon_{k^{\star}}\right)\right|=\Xi\left(w\right)$
if $w\in\theta_{u}\left(\mathcal{Q}_{a}\right)$ for some $a\geq\Xi\left(w\right)$
and some external co-ordinate $u$ of $\mathcal{Q}_{a}$ (but inside
$\theta_{s_{\Xi\left(w\right)-1}}\left(\mathcal{Q}_{\Xi\left(w\right)-1}\right)$,
where for convenience we set $\mathcal{Q}_{0}=\theta_{0}\left(\mathcal{Q}_{0}\right)=\mathcal{Q}_{n}$),
and $\left|E\left(\left\{ w\right\} ,\Upsilon{}_{k^{\star}}\right)\right|=\Xi\left(w\right)-1$
otherwise. Thus from (\ref{eq:winbstar}) it follows that in the former
case, $\mathcal{H}\left(\Upsilon_{k^{\star}}\cup\left\{ w\right\} \right)<\mathcal{H}\left(\Upsilon_{k^{\star}}\right)$
if and only if $\Xi\left(w\right)\geq\left\lceil \frac{n-\mathfrak{h}}{2}\right\rceil =\delta$,
while in the latter case $\mathcal{H}\left(\Upsilon_{k^{\star}}\cup\left\{ w\right\} \right)<\mathcal{H}\left(\Upsilon_{k^{\star}}\right)$
is not possible. But this implies that $w\in\theta_{s_{\delta-1}}\left(\mathcal{Q}_{\delta-1}\right)$
and $w$ is a neighbour of the vertex in $\mathcal{Q}_{\delta}$.
Since $\left\lceil n-\mathfrak{h}-2\left(\delta-1\right)\right\rceil \in\left\{ 1,2\right\} $,
if $\left\lceil n-\mathfrak{h}-2\left(\delta-1\right)\right\rceil =1$
there is a unique vertex that satisfies this (which implies $\Upsilon_{k^{\star}}\cup\left\{ w\right\} =\Upsilon_{k^{\star}+1}$),
and if $\left\lceil n-\mathfrak{h}-2\left(\delta-1\right)\right\rceil =2$,
there are two vertices in $\theta_{s_{\delta-1}}\left(\mathcal{Q}_{\delta-1}\right)\backslash\mathcal{Q}_{\delta}$
that satisfy this (one of which is again $\Upsilon_{k^{\star}+1}$).
Therefore,
\begin{lem}
Using the above notation,
\[
\mathscr{B}^{\star}=\left\{ \varphi\left(\Upsilon_{k^{\star}+1}\right):\,\varphi\,\mbox{is an isomorphism of }\mathcal{Q}_{n}\right\} 
\]
and 
\[
N^{+}\left(\sigma\right)=\left|\left\{ \sigma\prime\in\mathscr{B}^{\star}:\,\sigma\sim\sigma\prime\right\} \right|=\left\lceil n-\mathfrak{h}-2\delta+2\right\rceil 
\]

\medskip{}

\end{lem}
~
\begin{lem}
The value of $K$ in (\ref{eq:variationalformsimplified}) is given
by 
\begin{eqnarray*}
K & = & \left(\frac{1+\left\lceil n-\mathfrak{h}-2\delta+2\right\rceil }{\left\lceil n-\mathfrak{h}-2\delta+2\right\rceil }\right)/\left|\mathscr{C}^{\star}\right|\\
 & = & \frac{\left(n-\left\lceil n-\mathfrak{h}\right\rceil +1\right)!}{n!2^{n-4}\left(1+\left\lceil n-\mathfrak{h}-2\delta+2\right\rceil \right)}=\frac{\left\lceil \mathfrak{h}\right\rceil !}{n!2^{n-4}\left(3-\epsilon\right)}
\end{eqnarray*}
with $1-\left\lfloor n-\mathfrak{h}\right\rfloor \mbox{mod}2$. 
\end{lem}

\section{Stability levels and reference paths\label{sec:stablevelrefpath}}

The proof of Theorem \ref{lem:metastabset} is virtually identical
to the proof of the analogous problem on $\mathbb{Z}^{2}$, given
in chapter 17 in \cite{key-2}. It exploits translation invariance
in the underlying graph, and the possibility to initiate a uniformly
optimal path (as defined in the statement of Lemma \ref{lem:uniformoptimalpath})
starting from any vertex. 
\begin{proof}[Proof of Theorem \ref{lem:metastabset}]
 Let $\sigma\in\Omega$, $\sigma\notin\left\{ \boxminus,\boxplus\right\} $.
We will show that $\mathscr{V}_{\sigma}<\Gamma^{\star}$, which by
definition implies that $\Omega_{m}=\left\{ \boxminus\right\} $.
Pick any $w\in\sigma$ s.t. $\left(w,y\right)\in E_{n}$ for some
$y\in\overline{\sigma}$, and let $\gamma=\left(\gamma_{0},\ldots,\gamma_{2^{n}}\right)$
be an optimal path with initial steps $\gamma_{1}=\left\{ y\right\} $
and $\gamma_{2}=\left\{ w,y\right\} $ (this is always possible by
symmetry of the hypercube). Then 
\[
\sigma\cap\gamma_{1}=\boxminus
\]
and 
\[
1\leq\left|\sigma\cap\gamma_{k}\right|<k\quad\forall k\geq2
\]
Let us also denote by 
\[
k^{-}:=\min\left\{ i\,\vert\,\mathcal{H}\left(\gamma_{i}\right)\leq\mathcal{H}\left(\boxminus\right)\right\} 
\]
and note that by means of the following elementary observations, that
for any $A,B\subseteq V_{n}$
\begin{eqnarray*}
\left|E\left(A\cup B,\overline{A\cup B}\right)\right|+\left|E\left(A\cap B,\overline{A\cap B}\right)\right| & \leq & \left|E\left(A,\overline{A}\right)\right|+\left|E\left(B,\overline{B}\right)\right|\\
\left|A\cup B\right|+\left|A\cap B\right| & = & \left|A\right|+\left|B\right|
\end{eqnarray*}
it follows that for $1\leq i\leq k^{-}$

\begin{eqnarray*}
\mathcal{H}\left(\sigma\cup\gamma_{i}\right)-\mathcal{H}\left(\sigma\right) & = & \left(\left|E\left(\sigma\cup\gamma_{i},\overline{\sigma\cup\gamma_{i}}\right)\right|-\left|E\left(\sigma,\overline{\sigma}\right)\right|\right)-\mathfrak{h}\left(\left|\sigma\cup\gamma_{i}\right|-\left|\sigma\right|\right)\\
 & \leq & \left|E\left(\gamma_{i},\overline{\gamma_{i}}\right)\right|-\left|E\left(\sigma\cap\gamma_{i},\overline{\sigma\cap\gamma_{i}}\right)\right|-\mathfrak{h}\left(\left|\gamma_{i}\right|-\left|\sigma\cap\gamma_{i}\right|\right)\\
 & = & \mathcal{H}\left(\gamma_{i}\right)-\mathcal{H}\left(\gamma_{i}\cap\sigma\right)\\
 & < & \mathcal{H}\left(\gamma_{i}\right)-\mathcal{H}\left(\boxminus\right)
\end{eqnarray*}
where the last inequality follows from the fact that if $\left|\gamma_{i}\cap\sigma\right|=m$
for some $m<i$, and hence by uniform minimality of the sets $\gamma_{j}$
\[
\mathcal{H}\left(\gamma_{i}\cap\sigma\right)\geq\mathcal{H}\left(\gamma_{m}\right)>\mathcal{H}\left(\boxminus\right)
\]
This shows that $\mathscr{V}_{\sigma}<\Gamma^{\star}$.
\end{proof}
\newpage{}

\section{\label{sec:Appendix-A}Appendix A}

In this section we will show that if $W$ is not a good set (as defined
in Remark \ref{remgoodset}), $\left|E\left(W,\overline{W}\right)\right|$
is not minimal (as defined in Section \ref{sec:Isoperimetric}). Note
that this is equivalent to showing $\left|E\left(W,W\right)\right|$
is not maximal. And unlike $\left|E\left(W,\overline{W}\right)\right|$,
the quantity $\left|E\left(W,W\right)\right|$ is invariant of the
size of the cube in which $W$ is embedded. 

We start with a definition. We will say that a set $U\subseteq V_{n}$
with $2^{r}<\left|U\right|\leq2^{r+1}$ is \emph{well-contained} if
there is a $\left(r+1\right)$-dimensional sub-cube of $\mathcal{Q}_{n}$
containing $U$. Note that every set $U$ of size $\left|U\right|>2^{n-1}$
is well-contained. The following lemma shows that if $\left|E\left(W,\overline{W}\right)\right|$
is minimal, then $W$ must be well-contained.
\begin{lem}
\label{lem:notwellcontained}If $W$ is not well-contained, $\left|E\left(W,W\right)\right|$
is not maximal.\end{lem}
\begin{proof}
We begin with an observation: if $\mathcal{C}_{0}$ is a sub-cube
and $\mathcal{C}_{1}=\theta_{s}\left(\mathcal{C}_{0}\right)$ for
some external co-ordinate $s$ of $\mathcal{C}_{0}$ (recall this
means $\mathcal{C}_{0}$ and $\mathcal{C}_{1}$ are disjoint sub-cubes
of the same size, and there is some $1\leq s\leq n$ such that every
$u\in\mathcal{C}_{0}$ can be mapped to a $v\in\mathcal{C}_{1}$ by
changing the value at $u\left(s\right)$), and if $U_{0}\subseteq\mathcal{C}_{0}$,
$U_{1}\subseteq\mathcal{C}_{1}$ and $U=U_{0}\cup U_{1}$, then 
\begin{eqnarray}
\left|E\left(U,U\right)\right| & = & \left|E\left(U_{0},U_{0}\right)\right|+\left|E\left(U_{1},U_{1}\right)\right|+\left|E\left(U_{1},U_{0}\right)\right|\label{eq:Ewwexpansion}\\
 & \leq & \left|E\left(U_{0},U_{0}\right)\right|+\left|E\left(U_{1},U_{1}\right)\right|+\min\left(\left|U_{1}\right|,\left|U_{0}\right|\right)\nonumber 
\end{eqnarray}
where the inequality follows from the observation that every $v\in U_{1}$
has at most one neighbour in $U_{0}$, and vice versa. Furthermore, 

\medskip{}

\begin{minipage}[t]{1\columnwidth}%
\begin{claim*}
If $U$ is a good set, then the inequality in (\ref{eq:Ewwexpansion})
is an equality.\end{claim*}
\begin{proof}
Let $r$ be such that $2^{r}<\left|U\right|\leq2^{r+1}$. By definition,
there is some $l\leq r+1$ such that $U$ can be decomposed into $l$
disjoint sets 
\[
U=U^{1}\cup U^{2}\cdots U^{l}
\]
Here $U^{1}$ is the set of all vertices in some $a_{1}$-dimensional
sub-cube, with $a_{1}=r$, and $U^{i}$ ($i>1$) is the set of all
vertices in some $a_{i}$-dimensional sub-cube, with $a_{i}<a_{i-1}$.
Furthermore, $\cup_{j=i}^{n}U^{i}\subseteq\theta_{b_{i-1}}\left(U^{i-1}\right)$
for $i\geq2$ and some external co-ordinate $b_{i-1}$ of $U^{i-1}$.
Observe that 
\begin{equation}
\left|E\left(U_{1},U_{0}\right)\right|=\sum_{j,k}\left|E\left(U^{j}\cap\mathcal{C}_{0},U^{k}\cap\mathcal{C}_{1}\right)\right|\label{eq:claimsum}
\end{equation}

If $s$ is an external co-ordinate of the $r+1$-dimensional sub-cube
$\mathcal{C}_{0}\cup\mathcal{C}_{1}$, then one of $U_{1}$ and $U_{2}$
is empty and hence (\ref{eq:Ewwexpansion}) is an equality. Otherwise
let $\Xi:=\min\left\{ i\,:\, s\mbox{ is an external co-ordinate of }U^{i}\right\} $,
and suppose first that $s\neq b_{i}$ for $1\leq i\leq l-1$. Then
for each $j<\Xi$, $\left|U^{j}\cap\mathcal{C}_{0}\right|=\left|U^{j}\cap\mathcal{C}_{1}\right|=\frac{1}{2}\left|U^{j}\right|$
and this is also clearly equal to $\left|E\left(U^{j}\cap\mathcal{C}_{0},U^{j}\cap\mathcal{C}_{1}\right)\right|$.
Note also that one of $\left\{ \left|U^{j}\cap\mathcal{C}_{0}\right|\right\} _{j=\Xi}^{l}$,
$\left\{ \left|U^{j}\cap\mathcal{C}_{1}\right|\right\} _{j=\Xi}^{l}$
is a string of $0$'s, while the other is equal to to $\left\{ \left|U^{j}\right|\right\} _{j=\Xi}^{l}$,
and hence $\left|E\left(U^{j}\cap\mathcal{C}_{0},U^{j}\cap\mathcal{C}_{1}\right)\right|=0$.
And for any $j\neq k$, $\left|E\left(U^{j}\cap\mathcal{C}_{0},U^{k}\cap\mathcal{C}_{1}\right)\right|=0$
since elements differ both in co-ordinate $b_{j}$ and $s$. W.l.o.g.
assume that $\left\{ \left|U^{j}\cap\mathcal{C}_{1}\right|\right\} _{j=\Xi}^{l}=\left\{ 0\right\} _{j=\Xi}^{l}$,
so that $\min\left(\left|U_{1}\right|,\left|U_{0}\right|\right)=\left|U_{1}\right|=\frac{1}{2}\sum_{j=1}^{\Xi-1}\left|U^{j}\right|$.
But then also 
\begin{equation}
\left|E\left(U_{1},U_{0}\right)\right|=\sum_{j<\Xi}\left|E\left(U^{j}\cap\mathcal{C}_{0},U^{j}\cap\mathcal{C}_{1}\right)\right|=\frac{1}{2}\sum_{j=1}^{\Xi-1}\left|U^{j}\right|=\min\left(\left|U_{1}\right|,\left|U_{0}\right|\right)\label{claimsum2}
\end{equation}
which proves the claim in the case that $s\neq b_{i}$ for $1\leq i\leq l-1$.
If $s=b_{m}$ for some $m$, then (again w.l.o.g.) $U^{m}\subseteq\mathcal{C}_{0}$
and $\bigcup_{j>m}U^{j}\subseteq\mathcal{C}_{1}$ (so that $\min\left(\left|U_{1}\right|,\left|U_{0}\right|\right)=\left|U_{1}\right|$),
and we have that for every $v\in U^{j}$ there is a $w\in U^{m}$
such that $v=\theta_{b_{m}}\left(w\right)$. Thus 
\[
\left|E\left(U_{1},U_{0}\right)\right|=\sum_{j<\Xi}\left|E\left(U^{j}\cap\mathcal{C}_{0},U^{j}\cap\mathcal{C}_{1}\right)\right|+\sum_{j>m}\left|U^{j}\right|=\min\left(\left|U_{1}\right|,\left|U_{0}\right|\right)
\]
which proves the claim.\end{proof}
\end{minipage} 

\medskip{}

Let $r$ be such that $2^{r}<\left|W\right|\leq2^{r+1}$. We may assume
that $r+1\leq n-1$, since if $2^{n-1}<\left|W\right|$ then $W$
is by definition well-contained in the cube $\mathcal{Q}_{n}$. We
will start by induction on $n$. For $n=2$, the only sets that are
not well-contained are $W^{1}=\left\{ \left(0,0\right),\,\left(1,1\right)\right\} $
and $W^{2}=\left\{ \left(1,0\right),\,\left(0,1\right)\right\} $.
Clearly 
\begin{equation}
\left|E\left(W^{1},W^{1}\right)\right|=\left|E\left(W^{2},W^{2}\right)\right|=0\label{eq:n2notwellcontained}
\end{equation}
is not maximal. Now suppose that the statement of the lemma is true
whenever the setting is a hypercube of dimension less than or equal
to $n-1$, and let $W\subseteq V_{n}$ be a set that is not well-contained.
Let $W_{0}=\left\{ w\in W:\, w\left(1\right)=0\right\} $ with $W_{1}$
defined similarly, so that $W_{0}\cup W_{1}=W$, and suppose w.l.o.g.
that $\left|W_{0}\right|\geq\left|W_{1}\right|$. Note that the sets
$W_{0}$ and $W_{1}$ are contained in two disjoint hypercubes, call
them $\mathcal{Q}_{n-1}^{0}$ and $\mathcal{Q}_{n-1}^{1}$, of dimension
$n-1$. 

Let $r_{0}\leq n-2$ be such that $2^{r_{0}}<\left|W_{0}\right|\leq2^{r_{0}+1}$,
and define $r_{1}\leq r_{0}$ in a similar manner. If $W_{0}$ is
not well-contained, then by the inductive hypothesis $\left|E\left(W_{0},W_{0}\right)\right|$
is not maximal. Hence we can find a good set $\widetilde{W_{0}}$
in $\mathcal{Q}_{n-1}^{0}$ with $\left|W_{0}\right|=\left|\widetilde{W_{0}}\right|$
and $\left|E\left(W_{0},W_{0}\right)\right|<\left|E\left(\widetilde{W_{0}},\widetilde{W_{0}}\right)\right|$,
and we can also replace $W_{1}$ by a good set $\widetilde{W_{1}}$
of the same size such that $\left|E\left(W_{1},W_{1}\right)\right|\leq\left|E\left(\widetilde{W_{1}},\widetilde{W_{1}}\right)\right|$.
By (\ref{eq:Ewwexpansion}), $\left|E\left(W_{0},W_{1}\right)\right|\leq\left|W_{1}\right|$,
and we may take $\widetilde{W_{1}}$ such that $\left|E\left(\widetilde{W_{0}},\widetilde{W_{1}}\right)\right|=\left|\widetilde{W_{1}}\right|$
(by taking $\widetilde{W_{1}}$ to be a good subset of $\widetilde{W_{0}}$
with the first co-ordinate switched from $0$ to $1$), hence it also
follows that $\left|E\left(W_{0},W_{1}\right)\right|\leq\left|E\left(\widetilde{W_{0}},\widetilde{W_{1}}\right)\right|$.
By the expansion in (\ref{eq:Ewwexpansion}) it follows that the set
$\widetilde{W}:=\widetilde{W_{0}}\cup\widetilde{W_{1}}$ satisfies
$\left|E\left(W,W\right)\right|<\left|E\left(\widetilde{W},\widetilde{W}\right)\right|$,
and hence $\left|E\left(W,W\right)\right|$ is not maximal. The same
argument follows if $W_{1}$ is not well-contained. We may therefore
assume that $W_{0}$ and $W_{1}$ are well-contained. 

Suppose first that $r_{0}+1<n-1$. Assuming $W_{0}$ and $W_{1}$
are well-contained, we can find two disjoint sub-cubes $\mathcal{Q}_{r_{0}+1}^{0}$
and $\mathcal{Q}_{r_{1}+1}^{1}$ containing $W_{0}$ and $W_{1}$
respectively (they are disjoint since every vertex in $W_{0}$($W_{1}$)
has a $0$($1$) in its first co-ordinate, hence the same must be
true for every vertex in $\mathcal{Q}_{r_{0}+1}^{0}$($\mathcal{Q}_{r_{0}+1}^{1}$)).
We may also assume that $W_{1}$ is obtained from a subset of $W_{0}$
by switching the first co-ordinate to $1$, since otherwise $\left|E\left(W_{0},W_{1}\right)\right|<\min\left(\left|W_{0}\right|,\left|W_{1}\right|\right)$
and we can make the same argument as before to conclude that $\left|E\left(W,W\right)\right|$
is not maximal. It follows that $W$ is contained in a $\left(r_{0}+2\right)$-dimensional
sub-cube containing $\mathcal{Q}_{r_{0}+1}$ and $\mathcal{Q}_{r_{1}+1}$.
Since  $r_{0}+2\leq n-1$, it follows from the inductive hypothesis
that $\left|E\left(W,W\right)\right|$ is not maximal. 

Finally, if $r_{0}+1=n-1$, we can decompose $W_{0}$ into $W_{0}=W_{00}\cup W_{01}$,
with $W_{00}:=\left\{ w\in W_{0}:\, w\left(2\right)=0\right\} $ and
a similar definition for $W_{01}$. We can assume w.l.o.g. that $W_{0}$
and $W_{1}$ are good sets, since otherwise we can replace them by
good sets $\widetilde{W_{0}}$ and $\widetilde{W_{1}}$ as was done
in the previous case, to get $\left|E\left(W,W\right)\right|\leq\left|E\left(\widetilde{W},\widetilde{W}\right)\right|$,
where $\widetilde{W}=\widetilde{W_{0}}\cup\widetilde{W_{1}}$. Then
assuming $W_{0}$ is a good set, one of $W_{00}$, $W_{01}$ is the
set of all vertices of a $\left(n-2\right)$-dimensional sub-cube.
W.l.o.g. take this to be the set $W_{00}$, and note that $W_{01}$
is well-contained (since $W_{0}$ is a good set). Note that at least
one of the inequalities $\left|E\left(W_{00},W_{1}\right)\right|\leq\min\left(\left|W_{00}\right|,\left|W_{1}\right|\right)=\left|W_{1}\right|$
and $\left|E\left(W_{01},W_{1}\right)\right|\leq\min\left(\left|W_{01}\right|,\left|W_{1}\right|\right)$
is strict, since each $w\in W_{1}$ has at most one neighbour in $W_{0}$,
and that will be either in $W_{00}$ or $W_{01}$. Furthermore, we
can find a good set $W^{\dagger}$ of same size as $\widehat{W}:=W_{1}\cup W_{01}$
contained in the $\left(n-2\right)$-dimensional sub-cube that contains
$W_{01}$ such that $\left|E\left(W^{\dagger},W^{\dagger}\right)\right|\geq\left|E\left(\widehat{W},\widehat{W}\right)\right|$
and $\left|E\left(W^{\dagger},W_{00}\right)\right|=\left|W^{\dagger}\right|\geq\left|E\left(\widehat{W},W_{00}\right)\right|$.
But then at least one of the inequalities $\left|E\left(W^{\dagger},W^{\dagger}\right)\right|\geq\left|E\left(\widehat{W},\widehat{W}\right)\right|$
and $\left|E\left(W^{\dagger},W_{00}\right)\right|\geq\left|E\left(\widehat{W},W_{00}\right)\right|$
is strict, and hence $\left|E\left(W^{\dagger}\cup W_{00},W^{\dagger}\cup W_{00}\right)\right|>\left|E\left(W,W\right)\right|$.
It follows again $\left|E\left(W,W\right)\right|$ is not maximal.
\end{proof}
\medskip{}
 
\begin{proof}[Proof of Lemma \ref{lem:allminimalsets}]
As in the proof of Lemma \ref{lem:notwellcontained}, we will prove
the statement of this lemma by induction on the size of the main hypercube.
The case $n=2$ is simple, since the only sets that are not good are
the two sets $W^{1}$ and $W^{2}$ given in (\ref{eq:n2notwellcontained}).
Suppose now that whenever our setting is a hypercube of dimension
less than or equal to $n-1$, $U$ is not a good set implies $\left|E\left(U,U\right)\right|$
is not maximal. Let $W$ be a not-good subset of the $n$-dimensional
hypercube of size $2^{r}+k$ for $1\leq k\leq2^{r}$ and $0\leq r\leq n-1$.
Then $W$ falls under at least one of the following three cases:

1. There is no $\left(r+1\right)$-dimensional sub-cube which contains
the set $W$ (i.e. $W$ is not well-contained).

2. If $\mathcal{Q}_{r+1}$ is a $\left(r+1\right)$-dimensional sub-cube
of $\mathcal{Q}_{n}$ that contains $W$, and $\mathcal{Q}_{r+1}=\left(\mathcal{Q}_{r}^{0},\mathcal{Q}_{r}^{1}\right)$
is any decomposition of $\mathcal{Q}_{r+1}$ into two disjoint, $r$-dimensional
sub-cubes, then $\overline{W}\cap\mathcal{Q}_{r}^{0}\neq\emptyset$
and $\overline{W}\cap\mathcal{Q}_{r}^{1}\neq\emptyset$.

3. If $\mathcal{Q}_{r+1}$ is a $\left(r+1\right)$-dimensional sub-cube
of $\mathcal{Q}_{n}$ that contains $W$, and $\mathcal{Q}_{r+1}=\left(\mathcal{Q}_{r}^{0},\mathcal{Q}_{r}^{1}\right)$
is any decomposition of $\mathcal{Q}_{r+1}$ into two disjoint, $r$-dimensional
sub-cubes, then $\overline{W}\cap\mathcal{Q}_{r}^{0}=\emptyset$ implies
$W\cap\mathcal{Q}_{r}^{1}$ is not good.

The first case is covered by Lemma \ref{lem:notwellcontained}. The
third case follows almost immediately from the inductive hypothesis.
Indeed, if $W_{i}=W\cap\mathcal{Q}_{r}^{i}$ for $i\in\left\{ 0,1\right\} $,
then replacing $W_{1}$ by a good set $\widetilde{W_{1}}$ of the
same size and contained in $\mathcal{Q}_{r}^{0}$ implies that $\left|E\left(W_{1},W_{1}\right)\right|<\left|E\left(\widetilde{W_{1}},\widetilde{W_{1}}\right)\right|$
and $\left|E\left(W_{1},W_{0}\right)\right|\leq\left|W_{1}\right|=\left|E\left(\widetilde{W_{1}},W_{0}\right)\right|$.
Suppose now that $W$ falls under the second case. By the inductive
hypothesis, it follows that if $r+1<n$ or if either one of $W_{0}$,
$W_{1}$ is not good, $\left|E\left(W,W\right)\right|$ is not maximal.
Hence we may assume that $r+1=n$. But now we can consider the set
$U:=\overline{W}$ instead, since $\left|E\left(U,\overline{U}\right)\right|=\left|E\left(W,\overline{W}\right)\right|$.
Clearly $\left|U\right|<2^{n-1}$, hence again by the inductive hypothesis
we have that $\left|E\left(U,U\right)\right|$ is not maximal (and
hence $\left|E\left(U,\overline{U}\right)\right|$ is not minimal).
This proves that $\left|E\left(W,W\right)\right|$ is not maximal. 
\end{proof}

\section*{Acknowledgements}

This research is supported through NWO Gravitation Grant 024.002.003-NETWORKS.
The author would like to thank Frank den Hollander and Siamak Taati
for the helpful discussions and suggestions.

\end{document}